\documentclass[a4paper,11pt,leqno]{article}
\usepackage{authblk}
\usepackage[utf8]{inputenc}
\usepackage{amsmath,amsthm}
\usepackage[normalem]{ulem}
\usepackage{amssymb,latexsym}

\usepackage{mwe}
\usepackage{caption}
\usepackage{mathrsfs}
\DeclareMathOperator*{\esssup}{ess\,sup}

\newcommand{\R}{\mathbb{R}}

\newcommand{\dis}{\displaystyle}

\newtheorem{manualthminner}{Therorem}
\newenvironment{manualthm}[1]{%
	\renewcommand\themanualthminner{#1}%
	\manualthminner
}{\endmanualthminner}
\numberwithin{equation}{section}
\usepackage[
left=34mm,
right=34mm,
top=25mm,
bottom=25mm,
includeheadfoot,      
]{geometry}

\usepackage{xcolor}

\theoremstyle{plain}
\newtheorem{thm}{Theorem}[section]
\newtheorem{defi}{Definition}[section]

\newtheorem{remq}{Remark}[section]
\newtheorem{lem}{Lemma}[section]

\title{A class of fractional parabolic  reaction-diffusion systems {with control of {total} mass}: theory and numerics}
\author[1]{Maha Daoud }
\author[2]{El-Haj Laamri \thanks{Corresponding author: el-haj.laamri@univ-lorraine.fr}}
\author[1]{Azeddine Baalal}
\affil[1]{D\'epartement de Math\'ematiques et Informatique, 
	Facult\'e des Sciences A\"{i}n-chock, Universit\'e Hassan II, Casablanca  20100, Maroc}
\affil[2]{Institut Elie Cartan, Universit\'e de
	Lorraine, 54 506 Vandoeuvre-les-Nancy,
	France
	
	\medbreak
	
	{\normalsize  mahaadaoud@gmail.com, el-haj.laamri@univ-lorraine.fr, abaalal@gmail.com }}
\date{ \today}

\begin{document}
	\maketitle
	\begin{abstract}
	In this paper, we prove global-in-time existence of  strong solutions to a class of fractional 
	parabolic reaction-diffusion systems posed in a bounded domain of $\R^N$.
The nonlinear reactive terms are assumed to satisfy natural structure conditions which  provide non-negativity of the solutions and uniform control of the total mass. The diffusion operators are  of type $u_i\mapsto d_i(-\Delta)^s u_i$ where $0<s<1$.
 %
Global existence of strong solutions is proved under the assumption that the nonlinearities are at most of polynomial growth.
 Our results extend previous results obtained  when the diffusion operators are of type $u_i\mapsto -d_i\Delta u_i$.
 On the other hand,
 we use numerical simulations to examine the global existence of solutions to systems with exponentially growing right-hand sides, which remains so far an open theoretical question even in the case $s=1$.
	\bigbreak
	\medbreak
	\end{abstract}	
	
	\section{Introduction}
	
	The purpose  of this work is the study  of global existence in time of nonnegative strong  solutions to  fractional parabolic reaction-diffusion systems of the form:
	
	\begin{equation}\label{ReactiondiffusionSystem1}\tag{$S$}
		\left\{   \begin{array}{rcll}
			\forall i=1,\ldots,m,\hspace*{2.2cm}&&&\\
			\dis \partial_t u_i(t,\textbf{x})+d_i(-\Delta)^su_i(t,\textbf{x})&=&f_i(u_1(t,\textbf{x}),\ldots,u_m(t,\textbf{x})),&(t,\textbf{x})\in (0,T)\times\Omega,\\ \dis
			u_i(t, \textbf{x})&=&0,&(t,\textbf{x})\in (0,T)\times(\R^N\setminus\Omega),\\\dis
			u_i(0,\textbf{x})&=&u_{0i}(\textbf{x}),	&\textbf{x}\in\Omega,
		\end{array}
		\right.
	\end{equation}
	where $\Omega$ is a bounded regular open subset of $\R^N$ where $N\geq1$, 
	$0<s<1$, $m\geq 2$ and for each $i=1,\ldots m$,  the diffusion coefficient $d_i>0$ and $f_i$ is locally Lipschitz continuous.\\
	Here,  we mean by  $(-\Delta)^s$ the classical fractional Laplacian operator of order $s$  defined by		\begin{equation} \label{FractionalLaplacianDefinition}
		(-\Delta)^s\phi(\mathbf{x}):=a_{N,s} \text{\text{P.V.}} \int_{\mathbb{R}^{N}} \frac{\phi(\mathbf{x})-\phi(\mathbf{y})}{\|\mathbf{x}-\mathbf{y}\|^{N+2s}}d\mathbf{y},
	\end{equation}
	where $\|\cdot\|$ is the euclidean norm of $\R^N$, \text{P.V.} stands for the Cauchy principal value and $a_{N,s}$  is a normalization constant
	such that the following pair of identities :
	$$
	\lim _{s \rightarrow 0^{+}}(-\Delta)^{s} \phi=\phi 
	\quad\text{and}\quad
	\lim _{s \rightarrow 1^{-}}(-\Delta)^{s} \phi=-\Delta \phi
	$$
	holds (see, \textit{e.g.}, \cite[Proposition 4.4]{DiPalVal} and \cite[Proposition 2.1]{DaouLaam}). We refer readers not familiar with fractional laplacians to \cite{{DaouLaam}, {BisRadServ}, {DiPalVal}} and the references therein.

	Parabolic reaction-diffusion systems appear in various disciplines such as chemistry, biology and environment science (see~\cite{PierreSurvey,SchFis,Pis,Mur} and their bibliographies).	 
	%
	When governed by the classical Laplacian, such systems are related to normal diffusion processes.
	They have been extensively studied in the literature, see, for instance, \cite{PierreSurvey, LaamActa2011,FellLaam, LaamPert,LaamPierANHP, GoudVass10, PierSuzYam2019, SuzYam, CapGoudVass,  Soup18, FellMorgTang20, FellMorgTang20, AbdAttBenLaam} and the references therein.

	\medbreak
	
	\noindent However, in order to describe nonlocal diffusive processes, the classical Laplacian is no longer a well-suited model. Hence the interest of fractional Laplacian operators, which allow one to derive more advanced models, especially anomalous diffusion processes. 
	Additional insights can be found in \cite{AbaVal,Val,BucVal,Vaz1,Vaz2,HenLanStr,LisAl,DaouLaam}.	
	More generally, nonlocal partial differential equations (NPDEs) have come out for some years, and play a significant role in numerous scientific fields, as illustrated in \cite{Pod,KilSri} and their references.
	The fractional Laplacians are the prototypical example of nonlocal operators appearing in NPDEs.
	
	\medbreak
	
	\noindent Recently, many papers have investigated the existence of solutions to fractional reaction-diffusion problems consisting of a single parabolic equation ($m=1$) on open subsets of~$\R^N$. We refer the reader, for instance, to \cite{AndMazRosTol,GalWarma1,GalWarma2,LecMirRoq,BicWarZua2,BicWarZua3,FerRos}.
	
	\noindent On the other hand, there are few works dealing with \emph{systems} ($m\geq 2$) governed by fractional Laplacians, see~\cite{AtmaBirDaouLaam-Parabo, AtmaBirDaouLaam-Acta,AtmaBirDaouLaam-POT-GRAD, AtmaBirDaouLaam-s1-s2} and  ~\cite{{Ahmad+4Hindawi}, {Ahmad+Alsaedi+2}, {Alsaedi+Al-Yami+2}} in the case where $\Omega =\R^N$.
	 In this work, we are going to discuss the extensions of results known for $s=1$ or $m=1$ to the more general situation $0<s<1$ and $m\ge2$.
	
	\medbreak	
	
	Going back to System~\eqref{ReactiondiffusionSystem1}, we are only interested in nonnegative solutions, as in applications the unknown $\mathbf{u}=(u_1,\cdots,u_m)$  represents for example  concentrations of chemical species or population densities. 
	Therefore, the initial data have to be chosen nonnegative ; $u_{0i}\geq 0$ for each $i\in \{1,\cdots, m\}$. 
	It is  also well-known that the solutions (as long as they exist) remain nonnegative, provided that  the reaction  terms $f_i$ satisfy the so-called ``quasi-postivity" property, namely:
	\begin{equation}
		\forall 1\leq i\leq m,\quad f_i(r_1,\cdots,r_{i-1},0,r_{i+1},\cdots,r_m) \geq  0,
		\quad \forall \mathbf{r}=(r_1,\cdots,r_m) \in [0,+\infty)^m.
		\tag{{\bf P}}
	\end{equation}
	Furthermore, in many cases a control (sometimes even the preservation) of the total mass naturally follows the model, namely:
	\begin{equation}\label{ControlMass}
		\sum_{i=1}^m\int_\Omega u_i(t,\mathbf{x})\,d\mathbf{x} \leq C \sum_{i=1}^m\int_\Omega u_{0i}(\mathbf{x})\,d\mathbf{x} \quad\text{for some } C>0.
	\end{equation}
	The control (\ref{ControlMass})  is fulfilled if
	\begin{equation}
		\exists (a_1,\cdots a_m)\in (0,+\infty)^m \text{ such that } \dis\sum\limits_{i=1}^m a_if_i\leq 0,
		\tag{{\bf M}}
	\end{equation}
	or, more generally, if this sum grows at most linearly in its variables, 
	{{\it i.e.}
		\begin{equation}
			\forall \mathbf{r} \in[0,+\infty)^m,\; \sum\limits_{i=1}^m a_i f_i(\mathbf{r}) \leq C\Big[1+\sum\limits_{i=1}^m r_i\Big] \quad\text{for some } C\geq 0.
			\tag{{\bf M'}}
		\end{equation}
}	
	\medbreak
	
	\noindent Let us emphasize that properties 	{\bf (P)} and 	{\bf (M)} are satisfied in many applications, $e.g.$ in models describing evolution phenomena involving both spatial diffusion and chemical reactions. Nevertheless, there are some instances where the total mass of the solution is not controlled; for example : $m=2$, $f_1(u_1,u_2)=u_2^3$ and $f_2(u_1,u_2)=u_1^2$. 
	
	\bigbreak
	
	The question of global existence for reaction-diffusion systems has been a classical topic since the seventies, yet there are still many open and challenging problems. There are two major obstructions to the construction of global solutions. {\it Firstly}, when the right-hand sides $f_i$ have quadratic, even faster, growth for large values of the $u_i$'s. A  typical example  for three species is : 
	$$ f_1= \alpha_1g,\; f_2= \alpha_2g,\; f_3= -\alpha_3g \text{ where } g=u_3^{\alpha_3}-u_1^{\alpha_1}u_2^{\alpha_2} \text{ and } \alpha_1, \alpha_2, \alpha_3 \geq 1.$$
See~System \eqref{SystemChemReac-DiffCL} below for more precisions.	
	%
	{\it Secondly}, when the diffusion coefficients $d_i$ are very different, there is no {comparison principle}. Thus, no a priori estimates are available besides the $L^1$ control (\ref{ControlMass}).
	
	\smallbreak
	
	To put our work in context and highlight the novelty of this paper, let us briefly review the existing literature. 
	
	\medbreak
	
	\noindent $\bullet$ \underline{Case where the diffusions are driven by classical Laplacian ({\it i.e.} $s=1$)}
	
	\noindent It turns out that the structure \textbf{(P)}+ \textbf{(M)} does not keep the solution from blowing up in $L^\infty$-norm, even in finite time. 
	More precisely, the paper~\cite{PieSch2000} gives an explicit example of a two-species system fulfilling both \textbf{(P)} and \textbf{(M)}, with $d_1\neq d_2$ and strictly superquadratic polynomial nonlinearities, and such that 
	\begin{equation*}
		\exists T^*<+\infty,\quad \lim_{t\nearrow T^*}\Vert u_1(t, \cdot)\Vert_{L^\infty(\Omega)}= \lim_{t\nearrow T^*}\Vert u_2(t, \cdot)\Vert_{L^\infty(\Omega)}=+\infty.
	\end{equation*}
	Thus, in addition to the structure \textbf{(P)}+ \textbf{(M)}, some growth restrictions and extra structure on the nonlinearities are needed if one expects global existence of strong (or even weak) solutions.
	
	\smallbreak
	
	\noindent--- {\bf Strong solutions:}
	This issue has been intensively studied, especially when the initial data are bounded. Let us review some sufficient conditions on the $f_i$'s guaranteeing the global existence of a {\it strong} solution (see Definition \ref{Strong_solution}):
	\begin{itemize}
		\item[(i)] {\it triangular structure} (see~{\eqref{TriangularStructure-}}) and  polynomial growth. For more details, we refer to~\cite{PierreSurvey} and references therein. A  prototypical example of triangular structure when $m=2$ is $f_1\leq 0$ and $f_1+f_2\leq 0$;
		\item[(ii)] quadratic growth. See, for instance,~\cite{GoudVass10, PierSuzYam2019, CapGoudVass,  Soup18, FellMorgTang20, FellMorgTang21});
		\item[(iii)] diffusion coefficients close to one another, see, $e.g.$ \cite{CanDesFel, FellLaam} (sometimes called \emph{quasi-uniform} in some references, for example  in~\cite{FellMorgTang20, FellMorgTang21}); 
		\item[(iv)] when the growth of the~$f_i$'s is slightly stronger than polynomial, few results are known, and only for $m=2$. The most-studied model by far is $f_2(u_1,u_2)=-f_1(u_1,u_2)=u_1e^{u_2^\beta}$. In this case, the existence of strong solutions is established in \cite{HaraYouk} for $\beta<1$, and \cite{Barabanova} for $\beta=1$. 
		The article \cite{RebiBena} gives some extensions of those results. Surprisingly, in the three previous works, the global existence is established under a restrictive assumption on the size of the initial data $u_{01}$. On the other hand, the problem is still open for $\beta>1$.
	\end{itemize}
	\noindent--- {\bf Weak solutions:} By a {\it weak} solution, we 
	mean a solution in the sense of distributions or, equivalently here, in the sense of the variation-of-constant formula with the suitable semigroup (see Definition \ref{Weak_solution}). 
	Such weak $L^1$-solution had already been considered in \cite{PierreL1, LaamThese, BoudibaThese} to handle initial data in $L^1$. However, an extra condition of  {\it triangular structure} of the nonlinearities was required. Furthermore~\cite{Pierre3} (see also \cite[theorem 5.9]{PierreSurvey} for $m\geq 2$), if for some reason the $f_i$'s are bounded in $L^1((0,T)\times \Omega)$ for any $T$, then there exists a global weak solution for \emph{any} initial data in $L^1(\Omega)$, without any further assumption beyond non-negativity. This $L^1$~bound can be established under any one of the following assumptions:  
	\begin{enumerate}
		\item there are $m$ independent inequalities between the $f_i$'s, and the $u_{0i}$'s are just assumed to be in~$L^1(\Omega)$. For example, System~\eqref{ReactiondiffusionSystem1} where $m=2$ and $f_2(u_1,u_2)=-f_1(u_1,u_2)=u_1e^{u_2^\beta}$ has a weak solution for any $\beta>0$ and any initial data  in $L^1(\Omega)$. We refer to~\cite{Pierre3, PierreSurvey} (see also \cite{LaamPierANHP});
		\item  the $f_i$'s are of quadratic growth, and the $u_{0i}$'s belong to $L^2(\Omega)$, see \cite[theorem 5.11]{PierreSurvey};
		\item the $f_i$'s are super-quadratic, as it occurs in complex chemical reactions, and the $u_{0,i}$'s are in $ L^1(\Omega)\cap H^{-1}(\Omega)$. For more details, see \cite{LaamPert} (see also \cite{CiavoPert}).
	\end{enumerate}
	Concerning the elliptic case, the interested reader is referred to
	\cite{LaamPierANHP, LaamPierDCDS}.

	\medbreak
	
	To conclude this short summary, we would like to point out that the above references do not exhaust the rich literature on the subject. The interested reader can find an exhaustive review of the results known before 2010 in the extensive survey~\cite{PierreSurvey}, which also gives a general presentation of the problem, further references, and many deep comments on the mathematical difficulties raised by such systems. For some more recent results, see \cite{LaamActa2011, FellLaam, CanDesFel, PierSuzYam2019, QuitSoup, CapGoudVass,  Soup18, FellMorgTang20, FellMorgTang21} and their references.

	\medbreak
	
	\noindent $\bullet$ \underline{Case where the diffusions are driven by the fractional Laplacian ({\it i.e.} $0<s<1$):}\\
	Unlike the  case $s=1$, relatively little is known. To our knowledge, the existing works fall into two broad categories: either the domain $\Omega$ is bounded and the right-hand sides are of potential-gradient or gradient-gradient type; or $\Omega=\R^N$ and the r.h.s. are of polynomial or exponential growth.
	\begin{itemize}
		\item[$\diamond$] $\Omega$ bounded:
		in~\cite{AtmaBirDaouLaam-Parabo}, the first two authors and their coworkers investigated  the case of a $2\times2$~System where  the right-hand sides are $f_1(u_1,u_2)= \|\nabla u_2\|^q +h_1$ and $f_2(u_1,u_2)= \|\nabla u_1\|^p +h_2$ with $h_1, h_2\geq 0$. Obviously, $f_1$ and $f_2$ do not {belong to} the framework of~\eqref{ReactiondiffusionSystem1}, nor do they satisfy~{\bf(M)}.
		They also studied  the elliptic version of the system in the following two cases : 
		$$
		\begin{array}{ll}
		(S_1)&\quad\quad (-\Delta)^s u_1 = \|\nabla u_2\|^q+ \lambda h_1 \;,\; (-\Delta)^s u_2 =  \|\nabla u_1\|^{p}+ \mu h_2,\\\\
		(S_2)&\quad\quad(-\Delta)^s u_1 = u_2^{q}+ \lambda h_1 \;,\; (-\Delta)^s u_2 =  \|\nabla u_1\|^{p}+ \mu h_2,
		\end{array}
		$$  
		where $\lambda,\, \mu>0$.
		For more details, we refer the reader to~\cite{AtmaBirDaouLaam-Acta} for $(S_1)$  and~\cite{AtmaBirDaouLaam-POT-GRAD} for $(S_2)$; see also~\cite{AtmaBirDaouLaam-s1-s2} for a generalized System $(S_1)$ with two different {diffusions} {\it i.e.} $s_1\neq s_2$.
		%
		\item[$\diamond$] $\Omega=\R^N$: the only works that we know  are those of~\cite{{Ahmad+4Hindawi}, {Ahmad+Alsaedi+2}, {Alsaedi+Al-Yami+2}}.
	\end{itemize}
	
	\medbreak
	
	As far as we know, the question of global existence of solutions to System~\eqref{ReactiondiffusionSystem1}, {set in} a bounded domain~$\Omega$, with the reaction terms satisfying  {\bf (P)+(M)} and of polynomial or exponential growth, has not been addressed so far. This is the main goal of this work. More specifically:
	%
	\begin{enumerate}
\item   we will extend to the fractional case ($0<s<1$) two main known results  in the  classical case (\textit{i.e.} $s=1$)  and the reaction terms are polynomial growing.  The first one deals with the typical case of reversible chemical reactions for three species (see Theorem~\ref{MainTh1}). The second one deals with the case where the number of equations $m\geq 2$ and the $f_i$'s have a triangular structure (see Theorem~\ref{MainTh2}). 
\item we present some numerical simulations to examine the global existence of solutions to System~\eqref{ReactiondiffusionSystem1} in the case where $m=2$, $f_2(u_1,u_2)=-f_1(u_1,u_2)=u_1e^{u_2^\beta}$ {with  $\beta > 1$}. See Section~\ref{S_exp}.
\end{enumerate}	
%
{
 In order to prove our main theorems, we will  extend several tools known in the classical case to the fractional case such as the maximal regularity theorem (see Theorem~\ref{BoundedSol}),  a Lamberton-type estimation in $L^p$ (see Theorem~\ref{DualProblemIneq}), the so-called Pierre's duality Lemma (see Lemma~\ref{FractionalDualityThm}).
 Such results are interesting in themselves. 
Nevertheless, it should be noted that it is not always possible to extend any result from the classical case to the fractional case; and when it is possible, it is usually far from being trivial.
See, for instance,~\cite{AtmaBirDaouLaam-Parabo, AtmaBirDaouLaam-Acta,AtmaBirDaouLaam-POT-GRAD, AtmaBirDaouLaam-s1-s2} in  a different context. 
}

\medbreak

For the ease of reader and the sake of completeness, let us now recall those two results  that we will extend. 
\smallbreak
%
	%
	%

		%
	%

	%
	{
	\noindent$\bullet$ {\bf First known result. } Let us consider the following System 
		\begin{equation}\label{SystemChemReac-DiffCL} 
			\left\{
			\begin{array}{rclll}\partial_t u_1(t,\textbf{x})-d_1\Delta u_1(t,\textbf{x})&=&
	f_1(u_1(t,\textbf{x}),u_2(t,\textbf{x}),u_3(t,\textbf{x}))		, & (t,\textbf{x})\in Q_T,
			 \\
 \partial_t u_2(t,\textbf{x})-d_2\Delta u_2(t,\textbf{x})&=& f_2(u_1(t,\textbf{x}),u_2(t,\textbf{x}),u_3(t,\textbf{x}))
 , & (t,\textbf{x})\in Q_T, 
 \\
  \partial_t u_3(t,\textbf{x})-d_3 \Delta u_3(t,\textbf{x})&=& f_3(u_1(t,\textbf{x}),u_2(t,\textbf{x}),u_3(t,\textbf{x}))
  , & (t,\textbf{x})\in Q_T, \\ u_1(t,\textbf{x})=u_2(t,\textbf{x})=u_3(t,\textbf{x})&=&0,&(t,\textbf{x})\in (0,T)\times\textcolor{black}{\partial}\Omega, 
 \\ u_1(0,\textbf{x})&=&u_{01}(\textbf{x}) \geq 0, & \textbf{x} \in \Omega, \\ u_2(0, \textbf{x})&=&u_{02}(\textbf{x}) \geq 0, & \textbf{x} \in \Omega, \\ u_3(0, \textbf{x})&=&u_{03}(\textbf{x}) \geq 0, & \textbf{x} \in \Omega,\end{array}\right.
		\end{equation}
		where
		$$ f_1= \alpha_1g,\; f_2= \alpha_2g,\; f_3= -\alpha_3g \text{ with } g=u_3^{\alpha_3}-u_1^{\alpha_1}u_2^{\alpha_2} \text{ and } \alpha_1, \alpha_2, \alpha_3 \geq 1.$$
		%
%
%
%
\noindent That system naturally arises in chemical kinetics when modeling the  following  reversible reaction 
	\begin{equation}\label{reaction-chimique}
		\alpha_1 U_1+\alpha_2 U_2 \rightleftharpoons \alpha_3 U_3
		\end{equation}
		where $u_1,\, u_2,\,  u_3$ stand for the density of $U_1,\, U_2$ and $U_3$ respectively, {and $\alpha_1$, $\alpha_2$, $\alpha_3$ are the stoichiometric coefficients}.
		\medbreak
\noindent  Let us notice that in  addition to fulfilling {\bf (P)}, $f_1$, $f_2$ and $f_3$ also satisfy {\bf (M)}, namely  $$\alpha_2\alpha_3f_1+\alpha_1\alpha_3f_2+2\alpha_1\alpha_2f_3~=~0.$$
		\medbreak
		%
		%
\noindent Now, we are ready to state  the first  known result, 
(see 
\cite[Theorems 1 et 2]{LaamActa2011}).
\begin{thm} \label{Thm-ReversibleChemicalReaction-DiffCL}
Assume that
			$(u_{01},u_{02},u_{03}) \in (L^\infty(\Omega))^3$. System~\eqref{SystemChemReac-DiffCL} admits a unique nonnegative global strong solution in the following cases 
			\\
			{\em (i)}  $\alpha_1+\alpha_2<\alpha_3$ ;\\
			{\em (ii)} $ \alpha_3=1$ whatever are  $\alpha_1$ and $\alpha_2$ ; \\
			{\em (iii)} $d_1=d_3$ or $d_2=d_3$ whatever are $\alpha_1$, $\alpha_2$ and $\alpha_3$ ;	\\		
		{\em (iv)} $d_1=d_2\neq d_3$  for any $(\alpha_1,\alpha_2,\alpha_3)\in [1,+\infty)^2\times (1,+\infty)$ such that $\alpha_1+\alpha_2\neq\alpha_3$.			
	\end{thm}
	}	
	%
%
{
\noindent {\bf Comments:}\\ 
 Assume that the diffusion coefficients  $d_1,\,d_2$ and $d_3$  are very different. The question of global existence of a strong (even weak) solution to System \eqref{SystemChemReac-DiffCL}  in the case $2<\alpha_3 < \alpha_1+\alpha_2$ is widely open since at least 2011. This is very surprising from a chemical point of view because the  reaction~\eqref{reaction-chimique} is supposed to be reversible and therefore a strong solution should exist and be global as in the case $\alpha_1+\alpha_2< \alpha_3$.
}
More generally, that question is far from being understood. For example, the  global existence of a strong solution of the system modeling a reversible reaction of the form $\alpha_1 U_1+\alpha_2 U_2 \rightleftharpoons~\alpha_3 U_3+\alpha_4 U_4$ has only recently been proved and  in the sole case $\alpha_1=\alpha_2= \alpha_3=\alpha_4=1$, see~\cite{CapGoudVass,  Soup18, FellMorgTang20, FellMorgTang21}).

%
%
\medbreak
	\noindent$\bullet$ {\bf Second known result.} Introducing the classical counterpart of System~\eqref{ReactiondiffusionSystem1}, namely:
	\begin{equation}\label{ReactiondiffusionSystem2}
		\left\{   \begin{array}{rcll}
			\forall i=1,\ldots,m,\hspace*{1.45cm}&&&\\
			\dis\partial_t u_i(t,\textbf{x})-d_i\Delta u_i(t,\textbf{x})&=&f_i(u_1(t,\textbf{x}),\ldots,u_m(t,\textbf{x})),&(t,\textbf{x})\in (0,T)\times\Omega,\\ \dis
			u_i(t, \textbf{x})&=&0,&(t,\textbf{x})\in (0,T)\times\partial\Omega,\\\dis
			u_i(0,\textbf{x})&=&u_{0i}(\textbf{x}),	&\textbf{x}\in\Omega.
		\end{array}
		\right.
	\end{equation}
It holds that (see \cite[Theorem 3.5]{PierreSurvey} and its references, especially \cite{Morgan})
	\begin{thm} \label{Thm-TriangularStructure}
		Assume that for each $i\in\{1,\cdots,m\}$, $u_{0i}\in (L^\infty(\Omega))^+$ and   $f_i$ is at most of polynomial growth and satisfy $(\mathbf{P})$. Moreover, assume that  	
		\begin{equation}\label{TriangularStructure-}
			\left\{   \begin{array}{ll}
				\hskip-0.5mm\text{there \hskip-0.5mm exist \hskip-0.5mm  a vector } \mathbf{b} \in \mathbb{R}^m \text{and  \hskip-0.5mm a  \hskip-0.5mm lower triangular \hskip-0.5mm  invertible matrix} \hskip1.5mm Q \hskip-0.5mm\in \hskip-0.5mm M_m(\R^+) \text{ s.t. }\\
				\forall \mathbf{r}=(r_1,\ldots,r_m) \in[0,+ \infty)^m,\;\; Q\mathbf{f}(\mathbf{r}) \leq\Big[1+\sum\limits_{1 \leq i \leq m} r_i\Big] \mathbf{b}.
			\end{array}
			\right.
		\end{equation}	
		Then, System~\eqref{ReactiondiffusionSystem2} admits a  unique nonnegative global strong  solution.
	\end{thm}

	
	

	
	
	\smallbreak

	The rest of this article is set out in the following manner. In Section 2, we present the necessary background for the remaining sections. First, we set up some basic definitions about the notion of semigroups, and some required results concerning the associated evolution problem. Then, we recall some fundamental properties of the fractional laplacian within the framework of semigroup theory. In Section 3, we focus on the fractional evolution equation and examine  the conditions that ensure the boundedness of the solution. Furthermore, we concentrate on significant concepts such as the dual problem,  {the maximal regularity and the comparison principle}. 
{Section 4 is divided into three subsections. In subsection 4.1,   we first investigate the local existence of solutions to System \eqref{ReactiondiffusionSystem1}, then we prove the global existence in a particular case and we end it by extending Pierre's duality lemma to the fractional case. 
{Subsections 4.2 and 4.3 are devoted respectively to the proofs of our two main theorems. In Section 5, we present some numerical simulations to examine the global existence of solutions to a $2\times 2$~System with exponential growth (see \eqref{ExampleReactiondiffusionSystem2x2} in page \pageref{S_exp}), which remains a theoretical open question even in the case $s=1$. At last, we introduce the proofs of Theorems \ref{BoundedSol} and \ref{maximumprinciple} in the Appendix.
}
	
	\medbreak
	
	Prior to ending this section, le us fix some notations.
	
	\smallbreak 	\noindent{\bf Notations.}\quad 
	Throughout this paper, we will use the following standard notations :
	\begin{itemize}
		\item[---] $\Omega$ is a bounded regular open subset of $\mathbb{R}^N$ {with $N\geq1$}. 
		\item[---]  For any $T>0$,  $Q_T:= (0,T)\times \Omega$. 
		\item[---] $\|\cdot\|$ is the Euclidean norm of $\R^N$. 
		\item[---]For any $p\in[1,+\infty)$, $\|\phi\|_{L^p(\Omega)}=\left(\dis\int_\Omega|\phi(\mathbf{x})|^pd\mathbf{x} \right)^{\frac1p}$ and $\|\psi\|_{L^p(Q_T)}=\left(\dis\int_0^T\hskip-2mm\int_\Omega|\psi(t,\mathbf{x})|^p dt d\mathbf{x} \right)^{\frac1p}$.
		\item[---] $\|\phi\|_{L^\infty(\Omega)}= \esssup\limits_{\mathbf{x}\in\Omega} |\phi(\mathbf{x})|$ and  $\|\psi\|_{L^\infty(Q_T)}= \esssup\limits_{(t,\mathbf{x})\in Q_T} |\psi(t,\mathbf{x})|$.
		\item[---] 	$
		\mathbb{H}_{0}^{s}(\Omega):=\{\phi\in L^2(\mathbb{R}^{N})\; ;\; \|\phi\|_{\mathbb{H}^{s}_0(\Omega)}<+\infty\;\text{ and }\; \phi=0 \; \text{ in } \;\mathbb{R}^N\setminus \Omega\}
		$, where
		$$ \|\phi\|_{\mathbb{H}^{s}_0(\Omega)}:=\left(\iint_{\mathbb{R}^{N}\times\R^N}  \frac{|\phi(\textbf{x})-\phi(\textbf{y})|^{2}}{\|\textbf{x}-\textbf{y}\|^{N+2s}} d\textbf{x} d\textbf{y}\right)^{\frac{1}{2}}.$$
	\end{itemize}

	\section{Preliminaries}
	In this section, we briefly review some relevant results for later use. First, we recall  existence results of solutions to evolution problems associated with semigroups, along with a Lamberton-type estimation based on semigroups on $L^2$. Then, we discuss the main properties of the fractional Laplacian and the semigroup generated by it on $L^2$. 
	\subsection{Evolution problems associated with semigroups}	
	Let $(E,\Sigma,\mu)$ be a measure space. Let $A$ be the infinitesimal generator of a strongly continuous semigroup  $\{S_A(t)\}_{t\geq 0}$ on $L^2(E,\Sigma,\mu)$ (or simply $L^2(E)$). 
	We refer readers not familiar with the notion of semigroups to standard works such as \cite[Chapter~1]{Pazy} and \cite[Chapter 7]{Vrabie}.
	
	\smallbreak\noindent
	Now, let us associate with the operator $A$ the following evolution problem :
	\begin{equation}\label{ProblemOfLamberton}
		\left\{\begin{array}{rcll}
			\dfrac{d w}{dt}-A w&=&h, & \text{in } (0,T), \\
			w(0)&=& w_0, & 
		\end{array}\right.
	\end{equation}
	where $h(t)$ and $ w_0$ belong to $L^p(E)$, $p\geq 1$.
	\medbreak\noindent
	Before stating the existence theorems, let us make precise what we mean by \textit{strong} solution and \textit{weak} solution to Problem \eqref{ProblemOfLamberton}.
	\begin{defi}[Strong solution]\label{Strong_solution}
		A function $ w\;:\;[0, T) \rightarrow L^p(E)$ is called a strong solution to  \eqref{ProblemOfLamberton}  
		if :\\
		{\em(i)} $ w \in \mathcal{C}([0,T);L^p(E))\cap\mathcal{C}^1(0,T;L^p(E))$ ;\\
		{\em(ii)} for any $t\in (0,T)$, $ w(t)\in D(A) $ ; \\
		{\em(iii)} Problem \eqref{ProblemOfLamberton} is satisfied $a.e.$ on $[0,T)$.
	\end{defi}
	\begin{defi}[Weak solution] \label{Weak_solution}
		Let $h\in L^1(0,T; L^p(E))$. A function $ w\in \mathcal{C}([0,T];L^p(E))$ is called a \textit{weak} solution \textcolor{black}{to  \eqref{ProblemOfLamberton}} if it satisfies the following integral equation :
		\begin{equation}\label{weaksolution}
			w(t)=S_A(t) w_0+\int_0^t S_A(t-\tau) h(\tau) d \tau,\quad \forall t\in[0,T].
		\end{equation}
	\end{defi}
	\noindent The first existence Theorem reads as (see, for instance, \cite[Chapter 4]{Pazy}):
	\begin{thm}\label{PazyExistence}
		Let $h\in L^1(0,T; L^p(E))$ and $w_0\in L^p(E)$. Then, Problem \eqref{ProblemOfLamberton} admits a unique weak solution.
	\end{thm}
	\noindent Now, let us assume that $\{S_A(t)\}_{t\geq 0}$ satisfies the  two additional assumptions :
	
	\begin{itemize}
		\item[{\bf(H1)}]  $\{S_A(t)\}_{t\geq 0}$ is a bounded analytic semigroup on $L^2(E)$ ;
		\item[{\bf(H2)}]  for any $p\in[1,+\infty]$ and any $\phi\in L^p(E)\cap L^2(E)$, we have
		$$
		\forall t\geq 0,\quad \|S_A(t)\phi\|_{L^p(E)}\leq\|\phi\|_{L^p(E)}.
		$$
	\end{itemize}
	
	
	\noindent The following theorem provides a well-known result for estimating the weak solution to Problem \eqref{ProblemOfLamberton} in $L^p$ when $w_0=0$,  see \cite[Theorem 1]{Lamberton}.

	\begin{thm}[Lamberton-type estimation]\label{LambertonEst}
		Assume that $w_0=0$ and $h\in L^p((0,T)\times  E)$ with $p\in(1,+\infty)$. Let $w$ be the 
		weak solution to  Problem \eqref{ProblemOfLamberton}. Then, there exists a constant $C_p:=C(p)>0$ such that
		\begin{equation}\label{InegalitedeLamberton2}
			\left\| \frac{dw}{dt}\right\|_{L^p((0,T)\times  E)}+  \left\|Aw\right\|_{L^p((0,T)\times  E)}\leq C_p   \left\|h\right\|_{L^p((0,T)\times  E)}.
		\end{equation}
	\end{thm}

	\noindent We end this subsection by noting that when $h\equiv0$, we have $w(t)=S_A(t) w_0$ as the weak solution to \eqref{ProblemOfLamberton}. Hence, it is clear that not every weak solution to \eqref{ProblemOfLamberton} becomes a strong solution.
	%
%
	However, the two notions of solution coincide if $h$ satisfies a further condition, namely if  $h\in \mathcal{C}^1(0,T;L^p(E))$. More precisely, one has the next result (see, {\it e.g.}, \cite[Corollary 3.3]{Pazy}).  
	
	%
	
	\begin{thm} 
		Assume that $h\in L^1(0,T; L^p(E))\cap\mathcal{C}^1(0,T;L^p(E))$. Then, the 
		weak solution to Problem \eqref{ProblemOfLamberton} is a strong solution. \hfill\(\Box\) 
	\end{thm} 
	
	{
	{\bf From now on, we set $E=\Omega$}. As stated in the introduction,   $\Omega$ is assumed to be a bounded regular open subset of $\R^N$. 
	}

	\subsection{Semigroup generated by the Fractional Laplacian}
	

	It is well-known that the operator $A=-(-\Delta)^s$ with domain
	\begin{equation}\label{DA}
		D(A)=\{ \phi\in \mathbb{H}_{0}^{s}(\Omega),\; (-\Delta)^s\phi\in  L^2(\Omega)  \}
	\end{equation}
	generates a strongly continuous {{\it submarkovian}\footnote{A strongly continuous semigroup $\left\{S_A(t) \right\}_{t\geq 0}$  is said submarkovian if it satisfies the two following properties: \\
-- if $\phi\in L^2(\Omega)$ and $\phi\geq 0$, then $S_A(t)\phi\geq 0$ $\forall t\geq 0$~;\\
-- if $\phi\in L^\infty (\Omega)$, then $\|S_A(t)\phi\|_{L^\infty(\Omega)}\leq \|\phi\|_{L^\infty(\Omega)}$.
}} semigroup $\left\{S_A(t) \right\}_{t\geq 0}$ on $L^2(\Omega)$, see \cite[ Proposition 2.14]{GalWarma2} and \cite[Theorem 3.3]{ClausWarma}.

	\smallskip
	

	
\noindent The following theorem will be used throughout the rest of this paper. We refer, for instance, to \cite[Theorem 2.10]{ClausWarma} or \cite[Theorem 2.16]{GalWarma2} for further details.
		

	%

	%
	\begin{thm}\label{TheoremSemigroup} 
		Let
		$\{S_A(t)\}_{t\geq 0}$ be the submarkovian semigroup generated by  $A$ on $L^2(\Omega)$ with domain given by \eqref{DA}. Then:
		\\
		{\em(i)} the semigroup $\{S_A(t)\}_{t\geq 0}$ can be extended to a contraction semigroup on $L^p(\Omega)$ for any $p\in[1,+\infty]$. Moreover, each contraction semigroup is strongly continuous if $p\in[1,+\infty)$ and
		bounded analytic if $p\in(1,+\infty)$ ;\\
		{\em(ii)} the semigroup $\{S_A(t)\}_{t\geq 0}$ is ultracontractive, i.e., for any $\phi\in L^p(\Omega)$ and $1\leq~p\leq q\leq+\infty$, there exists a constant $C>0$ such that
		\begin{equation}\label{Ultracontractive} 
			\|S_A(t)\phi\|_{L^q(\Omega)}\leq C t^{-\frac{N}{2s}\left( \frac{1}{p}-\frac{1}{q}  \right)}\|\phi\|_{L^p(\Omega)}, \quad\forall t>0. 
		\end{equation}
	Moreover, if $p=q=+\infty$ and $C=1$, the semigroup $\left\{S_A(t)  \right\}_{t\geq 0}$ is said $L^\infty$-contractive.
		\hfill\(\Box\)
	\end{thm}
	%
	%
	%

%
	\section{Fractional evolution problem}
	In order to study the existence and regularity of solutions to our System, we will rely on properties of solutions to the equation.
To do so, let us consider the following problem~:
	\begin{equation}\label{ParabolicProblem}
		\left\{\begin{array}{rclll}
			\partial_t w(t,\textbf{x})+d (-\Delta)^s w(t,\textbf{x})&=&h(t,\textbf{x}), & (t,\textbf{x})\in Q_T, \\
			w(t,\textbf{x})&=&0, & (t,\mathbf{x})\in(0,T)  \times(\mathbb{R}^N\setminus\Omega), \\
			w( 0,\textbf{x})&=&w_{0}(\textbf{x})\geq 0, & \textbf{x}\in\Omega,
		\end{array}\right.
	\end{equation}
	where $d>0$, $w_0\in L^p(\Omega)$ and $h\in L^p(Q_T$), $p>1$. 
	\vskip2mm
	\noindent First, we recall the conditions under which the solution to Problem \eqref{ParabolicProblem} exists and is bounded. Then, we consider the dual problem of \eqref{ParabolicProblem} in order to prove a Lamberton-type 
	estimation in $L^p$. 
	%
	%
	\subsection{Existence and Boundedness}
	Now, let us denote $A:=-d(-\Delta)^s$. 
	According to Theorem \ref{TheoremSemigroup}, the operator $A$ verifies the assumptions {\bf(H1)} and {\bf(H2)}.
	Hence, by Theorem \ref{PazyExistence}, Problem \eqref{ParabolicProblem}
	admits a unique weak solution $w\in \mathcal{C}([0,T],L^p(\Omega))$ which fulfills
	\begin{equation}\label{UniqueMildSolution}
		w(t,\cdot)=S_{A}(t) w_0+\int_0^t S_{A}(t-\tau) h(\tau,\cdot) d\tau.
	\end{equation} 
	Furthermore by Theorem \ref{LambertonEst}, when $w_0=0$ there exists a constant $C_p> 0$ such that 
	\begin{equation}\label{InegalitedeLambertonLP}
		\left\| \partial_{ t} w\right\|_{L^p(Q_T)}+ d \left\|(-\Delta)^sw\right\|_{L^p(Q_T)}\leq C_p   \left\|h\right\|_{L^p(Q_T)}.
	\end{equation}

\noindent The regularity of the weak solution to Problem \eqref{ParabolicProblem} relies on the two Theorems presented below. 

%
%
\medbreak
\noindent--- The first Theorem outlines the  necessary conditions for the solution to Problem \eqref{ParabolicProblem} to have a maximal regularity property in $L^\infty$. 
	

	\begin{thm} \label{BoundedSol}
	Assume that  $w_0 \in L^\infty(\Omega)$ and $h\in L^p(Q_T$) with $p>1$.
	Let $w$ be the weak solution to Problem \eqref{ParabolicProblem}. Then, for any $p>\dfrac{N+2s}{2s}$, there exists a constant $C>0$ such that
		$$
		\|w\|_{L^\infty(Q_T)}\leq \|w_0\|_{L^\infty(\Omega)}+C\|h\|_{L^p(Q_T)}.
		$$
	\end{thm}
	
	{
\noindent ---	The second Theorem
is related to the {comparison principle} for Problem \eqref{ParabolicProblem}.
}
	
%
%
%
	\begin{thm}\label{maximumprinciple} 	Assume that  $w_0 \in L^\infty(\Omega)$  and $h\in L^p(Q_T$) with  $p>1$. Let $w$ be the weak solution to Problem \eqref{ParabolicProblem}. If $h\leq 0$, then $w\in L^\infty(Q_T)$ and 
		$$
		\|w\|_{L^\infty(Q_T)}\leq\|w_0\|_{L^\infty(\Omega)}.	
		$$
	\end{thm}
	
	{The two previous results are classical when $s=1$.  Surprisingly, we did not find, in the literature,  their equivalents in the fractional case. Therefore, we have  decided to introduce their proofs in this paper although they are  close to those of the $s=1$ case. Nevertheless, for reading fluency we preferred to postpone these proofs in the appendix.}

	\subsection{Fractional dual problem of Problem \eqref{ParabolicProblem}}

	In this subsection, we state and prove the existence and regularity result of solutions to the dual problem of  \eqref{ParabolicProblem}. Such result will be useful to prove 
	the global existence of solutions to our System.

	\noindent
   The dual problem of Problem \eqref{ParabolicProblem} is given by
	\begin{equation}\tag{$P_\varphi$}\label{DualProblem}
		\left\{\begin{array}{rclll}
			-\partial_t \mathcal{Z}(t,\textbf{x})+d(-\Delta)^s \mathcal{Z}(t,\textbf{x})&=&\varphi(t,\textbf{x}), & (t,\textbf{x})\in Q_T, \\
			\mathcal{Z}(t,\textbf{x})&=&0, & (t,\textbf{x})\in (0,T)\times(\R^N\setminus\Omega), \\
			\mathcal{Z}(T, \textbf{x})&=& 0, & \textbf{x}\in\Omega,
		\end{array}\right.
	\end{equation}
	where 
	$\varphi$ is a regular function.
	\medbreak
	%
	%
	{
\noindent The following theorem provides the conditions for the solution of the dual problem \eqref{DualProblem} to exist and verify a  Lamberton-type estimation in $L^p$.
}
{	\begin{thm} \label{DualProblemIneq} 
		Assume that $\varphi\in L^p(Q_T)$ with $p>1$. Then, Problem 
		\eqref{DualProblem} admits a unique weak solution. Moreover,
		there exists a constant $C:=C(p,T)>0$ such that    	\begin{equation}\label{DualityInequality}
			\left\| \partial_t \mathcal{Z}\right\|_{L^p(Q_T)}+  \|\mathcal{Z}\|_{L^p(Q_T)}+d \left\|(-\Delta)^s\mathcal{Z}\right\|_{L^p(Q_T)}+\|\mathcal{Z}_0\|_{L^p(\Omega)}\leq C   \left\|\varphi\right\|_{L^p(Q_T)},
		\end{equation}
		where $\mathcal{Z}_0:=\mathcal{Z}(0,\cdot)$.
\end{thm}   }
	\begin{proof}
		In Problem \eqref{DualProblem}, let us make the following change of variable : $t=T-\tau\Rightarrow \tau=T-t\in (0,T) $. Furthermore, let us denote $w(\tau,\cdot):=\mathcal{Z}(T-\tau,\cdot)$ and $h(\tau,\cdot):=\varphi(T-\tau,\cdot)$.\\
		Then, Problem \eqref{DualProblem} becomes
		\begin{equation}\label{DualProblem2}
			\left\{\begin{array}{rclll}
				\partial_{ \tau}\mathcal{Z}(T-\tau,\textbf{x})+d(-\Delta)^s \mathcal{Z}(T-\tau,\textbf{x})&=&\varphi(T-\tau,\textbf{x}), & (\tau,\textbf{x})\in Q_T, \\
				\mathcal{Z}(T-\tau,\textbf{x})&=&0, & (\tau,\textbf{x})\in (0,T)\times(\R^N\setminus\Omega), \\
				\mathcal{Z}(T, \textbf{x})&=& 0, & \textbf{x}\in\Omega,
			\end{array}\right.
		\end{equation}
		which is equivalent to
		\begin{equation}\label{ProblemedeLamberton2}
			\left\{\begin{array}{rclll}
				\partial_{\tau} w(\tau,\textbf{x})+d(-\Delta)^s w(\tau,\textbf{x})&=&h(\tau,\textbf{x}), & (\tau,\textbf{x})\in Q_T, \\
				w(\tau,\textbf{x})&=&0, & (\tau,\textbf{x})\in (0,T)\times(\R^N\setminus\Omega), \\
				w(0, \textbf{x})&=& 0, & \textbf{x}\in\Omega.
			\end{array}\right.
		\end{equation}
		Hence, the existence of a unique weak solution to \eqref{DualProblem} follows from Theorem \ref{ParabolicProblem}. In addition, there exists a constant $C_1:=C_1(p)> 0$ such that
		\begin{equation}\label{DualityInequality1}
			\left\| \partial_t\mathcal{Z}\right\|_{L^p(Q_T)}+  d \left\|(-\Delta)^s\mathcal{Z}\right\|_{L^p(Q_T)}\leq C_1   \left\|\varphi\right\|_{L^p(Q_T)}.
		\end{equation}   
		Now, let us integrate the first equation of \eqref{DualProblem} over $(0,T)$. We obtain
		\begin{equation}
			-\int_0^T\partial_t \mathcal{Z}(t,\textbf{x}) dt=\int_0^T\left(\varphi(t,\textbf{x})-d(-\Delta)^s \mathcal{Z}(t,\textbf{x})\right)dt.
		\end{equation}   
		For any $\textbf{x}\in\Omega$, $\mathcal{Z}(T,\textbf{x})=0$ . Then,
		\begin{equation}
			\mathcal{Z}_0(\textbf{x})=\int_0^T\left(\varphi(t,\textbf{x})-d(-\Delta)^s \mathcal{Z}(t,\textbf{x})\right)dt.
		\end{equation}    
		Therefore, we get
		\begin{equation}\left\{  
			\begin{array}{lll}
				\|\mathcal{Z}_0\|_{L^p(\Omega)}&\displaystyle= \left\|\int_0^T(\varphi(t,\cdot)-d(-\Delta)^s \mathcal{Z}(t,\cdot))dt\right\|_{L^p(\Omega)}&
				\\&\displaystyle\leq \int_0^T\left\|\varphi(t,\cdot)-d(-\Delta)^s \mathcal{Z}(t,\cdot)\right\|_{L^p(\Omega)}dt
				&
				\\&\displaystyle\leq \int_0^T\left\|\varphi(t,\cdot)\right\|_{L^p(\Omega)}dt  +\int_0^T\left\|  d(-\Delta)^s \mathcal{Z}(t,\cdot)\right\|_{L^p(\Omega)}dt 
				&   \\&\displaystyle\leq T^{1-\frac{1}{p}}\left(\|\varphi\|_{L^p(Q_T)}  +d\|  (-\Delta)^s \mathcal{Z}\|_{L^p(Q_T)}\right) 
				&  \\&\displaystyle\leq T^{1-\frac{1}{p}} \max(1,C_1) \|\varphi\|_{L^pQ_T)}. & 
			\end{array}
			\right.
		\end{equation}
	The last inequality follows by the estimate \eqref{DualityInequality1}.   
		Hence,
		\begin{equation}\label{DualityInequality2}
			\|\mathcal{Z}_0\|_{L^p(\Omega)}\leq C_2 \|\varphi\|_{L^p(Q_T)},
		\end{equation}
		where $C_2:=C_2(p,T)> 0$.\\
		As $\mathcal{Z}$ is the weak solution to Problem \eqref{DualProblem}, thus it verifies the following integral equation :
		\begin{equation}
			\mathcal{Z}(t,\textbf{x})=\int_t^T S_{A}(\tau-t)\varphi(\tau,\textbf{x}) d\tau,\quad (t,\textbf{x})\in [0,T]\times \Omega.
		\end{equation}
		Thus, we have
		\begin{equation}
			\left\{  
			\begin{array}{lll}
				\|\mathcal{Z}(t,\cdot)\|_{L^p(\Omega)}&\leq  \displaystyle \left\|   \int_t^T S_{A}(\tau-t)\varphi(\tau,\cdot) d\tau \right\|_{L^p(\Omega)}&\\&\leq \displaystyle \int_t^T \|S_{A}(\tau-t)\varphi(\tau,\cdot)\|_{L^p(\Omega)} d\tau &  \\&\leq \displaystyle \int_0^T \|S_{A}(\tau-t)\varphi(\tau,\cdot)\|_{L^p(\Omega)} d\tau&\\ &\leq \displaystyle C \int_0^T \|\varphi(\tau,\cdot)\|_{L^p(\Omega)} d\tau & 
				\\&\leq \displaystyle C T^{1-\frac{1}{p}} \|\varphi\|_{L^p(Q_T)}.& 
			\end{array}
			\right.\end{equation} 
		Consequently, \begin{equation}		
			\displaystyle\|\mathcal{Z}\|_{L^\infty(0,T;L^p(\Omega))}\leq  C T^{1-\frac{1}{p}} 
			\|\varphi\|_{L^p(Q_T)}.
		\end{equation}
		As $L^\infty(0,T;L^p(\Omega))\hookrightarrow L^p(Q_T)$, we obtain the estimate \begin{equation}\label{DualityInequality3}
			\displaystyle\|\mathcal{Z}\|_{L^p(Q_T)}\leq  C_3 \|\varphi\|_{L^p(Q_T)},
		\end{equation} where $C_3:=C_3(p,T)> 0$.  \\
		Finally, let us denote $C:=\max(C_1,C_2,C_3)> 0$. From the estimates \eqref{DualityInequality1}, \eqref{DualityInequality2} and \eqref{DualityInequality3}, we obtain \eqref{DualityInequality}.
	\end{proof}
	\begin{remq} 
		Let us mention another notion of solution often found in the PDE literature.
		\\
		Let $\beta \in(0,1)$. It is well known (see \cite{LPPS, FelsKass}) that if $\varphi \in L^{\infty}\left(Q_T\right) \cap \mathcal{C}^{0, \beta}\left(Q_T\right)$, Problem \eqref{DualProblem} admits a regular solution $\mathcal{Z} \in L^{\infty}\left(Q_T\right)$ and the equation $-\partial_t \mathcal{Z}+(-\Delta)^s \mathcal{Z}=\varphi$ is satisfied in a pointwise sense.\\
		Let us multiply \eqref{UniqueMildSolution} by $\mathcal{Z}$ and integrate over $Q_T$. Then, we obtain
		\begin{equation}\label{WeakSolution}
			\iint_{Q_T} w(t,\mathbf{x})\left(-\partial_t\mathcal{Z}(t,\mathbf{x})+(-\Delta)^s \mathcal{Z}(t,\mathbf{x})\right) d\mathbf{x} d t=\iint_{Q_T} h(t,\mathbf{x}) \mathcal{Z}(t,\mathbf{x}) d\mathbf{x} d t+\int_{\Omega} w_0(\mathbf{x}) \mathcal{Z}_0(\mathbf{x}) d\mathbf{x}.
		\end{equation}
		A solution $\mathcal{Z}\in\mathcal{C}([0,T];L^1(\Omega))$ that satisfies \eqref{WeakSolution} is called a  solution in the sense of
			distributions to Problem \eqref{ParabolicProblem}. The existence and uniqueness criteria for such solutions have been established in~\cite[Theorem~28]{LPPS} by using an approximation argument.	Due to the properties of the semigroup generated by the fractional Laplacian, this notion coincides with that of a weak solution.
		\hfill\(\Box\)
	\end{remq}

	
		\section{Main results of global existence for Fractional Reaction-Diffusion Systems}
		
	Let us recall System 
		\begin{equation}\label{ReactiondiffusionSystem} \tag{$S$}
	\left\{   \begin{array}{rcll}
		\forall i=1,\ldots,m,\hspace*{2.2cm}&&&\\
		\dis \partial_t u_i(t,\textbf{x})+d_i(-\Delta)^su_i(t,\textbf{x})&=&f_i(u_1(t,\textbf{x}),\ldots,u_m(t,\textbf{x})),&(t,\textbf{x})\in Q_T,\\ \dis
		u_i(t, \textbf{x})&=&0,&(t,\textbf{x})\in (0,T)\times(\R^N\setminus\Omega),\\\dis
		u_i(0,\textbf{x})&=&u_{0i}(\textbf{x}),	&\textbf{x}\in\Omega,
	\end{array}
	\right.
\end{equation}
where $m\geq 2$, $0<s<1$ and for each $i=1,\ldots m$, $d_i>0$, $u_{0i}\in L^\infty(\Omega)$ and 
$f_i$ is locally Lipchitz continuous.
		
\medbreak	
	
\noindent As mentioned in the introduction, our main contribution is to extend Theorem~\ref{Thm-ReversibleChemicalReaction-DiffCL} and Theorem~\ref{Thm-TriangularStructure} to the  fractional case. That is the purpose of this section.  Moreover, for the sake of  clarity, this section will be divided into three subsections. In the first one, we will treat the local existence and extend Pierre's duality lemma to the fractional case. Subsections 4.2 and 4.3 will be devoted respectively to the proofs of our main Theorems~\ref{MainTh1} and~\ref{MainTh2}.\\
%
{ One should not {conclude} this short introduction without mentioning that  once we have established the necessary tools adapted to the fractional Laplacian, the demonstrations are similar to those for the case where the diffusions are governed by the classical Laplacian. However, for the convenience of the reader and for the completeness of the article, we will give {the proofs} in detail.
}


		 

 \subsection{Local existence and extension of Pierre's duality lemma to the fractional case}
{
\noindent$\bullet$ {\bf Local existence.} The local existence lemma  is classical.   Let us recall here its statement. As the proof is straightforward, we omit it.
\begin{lem}\label{LocalExistence}
 	Assume that
 	$\mathbf{u}_0:=(u_{01},\ldots,u_{0m}) \in (L^\infty(\Omega))^m$ , the $f_i$'s  are  locally Lipschitz continuous. Then,
	  there exists $T_{\max}>0$ and 
	$\Phi=(\varphi_1,\ldots,\varphi_m)\in \mathcal{C}([0,T_{\max}),[0, +\infty)^m)$ such that :\\
	{\em (i)} System~\eqref{ReactiondiffusionSystem1} has  a unique nonnegative strong solution $\mathbf{u}=(u_{1},\ldots,u_{m})$ in $(0,T_{\max})\times \Omega$; \\
{\em (ii)}	for each $i\in \{1,\ldots,m\}$
\begin{equation}
 \| u_i(t,\cdot)\|_{L^\infty(\Omega)}\leq \varphi_i(t) \text{ for any } t\in(0,T_{\max});
\end{equation}
{\em (iii)}	if $T_{\max}<+\infty$, then $\dis\lim_{t\nearrow T_{\max}}\dis\sum\limits_{i=1}^m \|u_i(t,\cdot)\|_{L^\infty(\Omega)}=+\infty$;\\
{\em (iv)}	if, in addition,  the $f_i$'s satisfy {\bf (P)}, then 
$$\big(\forall i \in\{1,\ldots,m\},\; u_{0i}(.)\geq 0\big) 	\Longrightarrow \big(\forall i \in\{1,\ldots,m\},\; u_i(t,.)\geq 0 \; \forall t\in [0,T_{\max})\big).$$
\end{lem}
%
%
\begin{remq}
 According to {\em (iii)}, in order  to prove global existence of strong solutions to System~\eqref{ReactiondiffusionSystem1}, it is sufficient to prove an a priori  {estimate} of the form 
\begin{equation}\label{crucial-estimate}
 \forall t\in [0,T_{\max}),\; \dis\sum\limits_{i=1}^m \|u_i(t,\cdot) \|_{L^\infty(\Omega)} <\psi(t),  
 \end{equation}
 where $\psi :[0,+\infty)\rightarrow [0,+\infty)$ is a continuous function.
\end{remq}
}

 \medbreak
{
\noindent $\bullet$ {\bf Proof of global existence in a specific case} 
}
 
 It turns out that an estimate like  (\ref{crucial-estimate}) is far from being obvious for our System except in the trivial case where  $d_1=\ldots=d_m=d$ and the $f_i$'s satisfy {\bf (P)}.
\\
Then, suppose that for some $(a_1,\ldots a_m)\in (0,+\infty)^m$ such that 	
\begin{equation}\label{M-et-Mprime}
	\forall \mathbf{r} \in[0,+\infty)^m,\; \sum\limits_{i=1}^m a_i f_i(\mathbf{r}) \leq C\Big[1+\sum\limits_{i=1}^m r_i\Big], \quad\text{where } C\geq 0.
\end{equation}
 Let $T\in(0,T_{\max})$, $t\in (0,T]$ and $\mathbf{x}\in\Omega$. We denote $W(t,\mathbf{x}):=\dis\sum\limits_{i=1}^m a_i u_i(t,\mathbf{x})$ and $\textcolor{black}{W_0(\mathbf{x})}:=W(0,\mathbf{x})$.
 Let us multiply each $i$th equation of \eqref{ReactiondiffusionSystem} by $a_i$ and sum over $i$. Thus, we get
 \begin{equation}\label{Problem2}
\partial_t W(t,\mathbf{x})+d(-\Delta)^s W(t,\mathbf{x})=\dis\sum\limits_{i=1}^m a_i f_i(u_1(t,\mathbf{x}),\ldots,u_m(t,\mathbf{x})).
 \end{equation} 
 \smallbreak
 --- If $C=0$ in \eqref{M-et-Mprime} ({$i.e.$} we have the assumption {\bf (M)}), we deduce from  Theorem \ref{maximumprinciple} that:
 $$\|W\|_{L^\infty(Q_T)}\leq \|W_0\|_{L^\infty(\Omega)}$$
 which yields, thanks the nonnegativity of the $u_i$'s, to 
 \begin{equation}\label{Majoration-in-QT}
 \sum_{i=1}^m\|u_i\|_{L^\infty(Q_T)}\leq \frac{1}{a_{i_0}}\|W_0\|_{L^\infty(\Omega)} \text{ where } a_{i_0}:=\min\{a_1,\cdots,a_m\}.
 \end{equation}
 
 Given that \eqref{Majoration-in-QT} implies \eqref{crucial-estimate}, this concludes the proof in this case.
 \smallbreak
 --- If $C>0$ in \eqref{M-et-Mprime} ({i.e.} we have the assumption {\bf (M')}), introduce 
 $w$ the weak solution to
 \begin{equation}\label{Problem1}
 	\left\{  
 	\begin{array}{rcll}
 		\partial_t w(t,\mathbf{x})+d (-\Delta)^s w(t,\mathbf{x}) &=& C \Big[1+\dis\sum\limits_{i=1}^m  u_i(t,\mathbf{x})\Big], &(t,\mathbf{x})\in Q_T,\\
 		w(t,\mathbf{x})&=&0, &(t,\mathbf{x})\in (0,T)\times(\R^N\setminus\Omega),\\
 		w(0,\mathbf{x})&=&0,& \mathbf{x}\in\Omega.
 	\end{array}
 	\right.
 \end{equation}	
 Then, $w$ satisfies 
 \begin{equation}\label{MildSolw}
 	w(t,\mathbf{x})=C\dis\int_0^t S_i(t-\tau)\Big[1+\dis\sum\limits_{i=1}^m  u_i(\tau,\mathbf{x})\Big]d\tau,\quad (t,\mathbf{x})\in [0,T)\times\Omega.
 \end{equation}
Let us multiply each $i$th equation of \eqref{ReactiondiffusionSystem} by $a_i$ and sum over $i$. Thus, we get
 \begin{equation}\label{Problem2}
\partial_t W(t,\mathbf{x})+d(-\Delta)^s W(t,\mathbf{x})=\dis\sum\limits_{i=1}^m a_i f_i(u_1(t,\mathbf{x}),\ldots,u_m(t,\mathbf{x})).
 \end{equation}
By substracting the first equation of \eqref{Problem1} from \eqref{Problem2}, we obtain 
\begin{equation}\label{Problem3}
	\partial_t (W-w)(t,\mathbf{x})+d(-\Delta)^s (W-w)(t,\mathbf{x})\leq 0.
\end{equation}
	By Theorem \ref{maximumprinciple} , we get	
\begin{equation}\label{Problem4}
	\|W-w\|_{L^\infty(Q_T)}\leq \|\textcolor{black}{W_0}\|_{L^\infty(\Omega)}.
\end{equation}
Therefore,
   \begin{equation}\label{Problem5}
   	\|W\|_{L^\infty(Q_T)}\leq \|\textcolor{black}{W_0}\|_{L^\infty(\Omega)} + \|w\|_{L^\infty(Q_T)}.
   \end{equation}
By taking the $L^\infty(\Omega)$-norm of \eqref{MildSolw} and using the $L^\infty$-contractivity of $\{S_A(t)\}_{t\geq0}$, we obtain
$$
\textcolor{black}{a_{i_0}}\Big\|\dis\sum\limits_{i=1}^m  u_i(t,\cdot) \Big\|_{L^\infty(\Omega)} \leq  \Big\|\dis\sum\limits_{i=1}^m a_i u_{0i}\Big\|_{L^\infty(\Omega)} + C \int_0^t \Big[1+\Big\|\dis\sum\limits_{i=1}^m  u_i(\tau,\cdot) \Big\|_{L^\infty(\Omega)} \Big]d\tau .
$$	
Now, by applying { Gronwall's Lemma, we deduce that $$	\Big\|\sum\limits_{i=1}^m	u_i(t,\cdot)\Big\|_{L^\infty(\Omega)}\leq g(t),$$ where $g:[0,+\infty)\to [0,+\infty)$ is a continuous function. As {\bf (P)} assures us the nonnegativity of $u_i$ for any $i\in \{1,\cdots,m\}$, then $u_i(t,\cdot)$ is bounded in $L^\infty(\Omega)$ $\forall t\in[0,T_{\max})$.} This implies that $T_{\max}=+\infty$.$\square$

\medbreak
{
At this point,  the following  remark is in order. 
\begin{remq}
As we have just proved in the case $d_1=\cdots=d_m$, System (S) admits a global strong solution \textbf{whatever the growth} of the $f_i$'s. However,  the situation is quite more complicated if the diffusion coefficients are different from each other. Some additional assumptions on  the growth  of the source terms are needed if one expects global existence of strong  solutions even if $s=1$ as already specified in the introduction.
\end{remq}
}
\medbreak

\noindent$\bullet$ {\bf Extension of Pierre's duality lemma to the fractional case}

\smallbreak

\noindent 
 One of the main tools that we will use to establish our global existence theorems is the following lemma. 
This is nothing other than an extension of Pierre's duality lemma to the fractional case.
	\begin{lem}\label{FractionalDualityThm}
	Given $T>0$. Assume that $m=2$, $f_1$ and $f_2$ satisfy $(\mathbf{P})$ and $(\mathbf{M})$. Let $(u_1,u_2)$ be the strong solution to~\eqref{ReactiondiffusionSystem} in $Q_T$. Then, there exists a constant $C=C(p, T, \Omega)> 0$ such that
		$$
		\|u_2\|_{L^p\left(Q_T\right)} \leq C\left(1+\|u_1\|_{L^p\left(Q_T\right)}\right)\quad \forall 1<p<+\infty.
		$$
	\end{lem}
	\noindent Pierre's duality Lemma was first  stated in \cite{HolMarPie} and then widely exploited, see for example
	\cite{LaamThese,LaamActa2011, FellLaam, CanDesFel, PierreSurvey}  and the references therein. 
	
	\bigbreak
	
\begin{proof} Without loss of generality, we will assume that $a_1=a_2=1$.\\
	 Let $T\in(0,T_{\max})$ and $t\in (0,T]$. By summing the two equations of $u_1$ and $u_2$, we obtain
	 \begin{equation}\label{inequ_1+u_2}
	 \partial_t(u_1(t,\mathbf{x})+u_2(t,\mathbf{x})) + (-\Delta)^s(d_1u_1(t,\mathbf{x})+d_2u_2(t,\mathbf{x}))\leq 0,\quad (t,\mathbf{x})\in Q_T.
	 \end{equation}
 	Let $\varphi$ be a nonnegative regular function and let $\mathcal{Z}$ be the solution to Problem \eqref{DualProblem} with $d=d_2$.
	 Now, we multiply \eqref{inequ_1+u_2} by $\mathcal{Z}$ and integrate over $Q_T$. We get 
	 \begin{equation}
	 	\left\{
	 	\begin{array}{ll}
	 		\dis
	 	-\iint_{Q_T} (u_1(t,\mathbf{x})+u_2(t,\mathbf{x})) \partial_t\mathcal{Z}(t,\mathbf{x}) d\mathbf{x}dt+d_2 \iint_{Q_T} u_2(t,\mathbf{x}) (-\Delta)^s\mathcal{Z}(t,\mathbf{x}) d\mathbf{x}dt \\\\\dis \leq
	 	\int_\Omega (u_{01}-u_{02})(\mathbf{x}) \mathcal{Z}_0(\mathbf{x})d\mathbf{x}-d_1 \iint_{Q_T} u_1(t,\mathbf{x}) (-\Delta)^s\mathcal{Z}(t,\mathbf{x}) d\mathbf{x}dt .
 	\end{array}
 	\right.
	 \end{equation}	 
	 Then,
	  \begin{equation}
	 	\left\{
	 	\begin{array}{ll}
	 		\dis
	 		\iint_{Q_T} u_2(t,\mathbf{x}) \big[ -\partial_t\mathcal{Z}(t,\mathbf{x})+d_2 (-\Delta)^s\mathcal{Z}(t,\mathbf{x})  \big] d\mathbf{x}dt - \iint_{Q_T}
	 		u_1(t,\mathbf{x}) \partial_t\mathcal{Z}(t,\mathbf{x})  d\mathbf{x}dt \\\\\dis \leq
	 		\int_\Omega (u_{01}-u_{02})(\mathbf{x}) \mathcal{Z}_0(\mathbf{x})d\mathbf{x}-d_1 \iint_{Q_T} u_1(t,\mathbf{x}) (-\Delta)^s\mathcal{Z}(t,\mathbf{x}) d\mathbf{x}dt .
	 	\end{array}
	 	\right.
	 \end{equation}	 
	 Therefore,
	 \begin{equation}
	 	\iint_{Q_T} u_2(t,\mathbf{x}) \varphi(t,\mathbf{x}) d\mathbf{x}dt \leq \int_\Omega (u_{01}-u_{02})(\mathbf{x}) \mathcal{Z}_0(\mathbf{x})d\mathbf{x} +  \iint_{Q_T}
	 	u_1(t,\mathbf{x})\big[ \partial_t\mathcal{Z}(t,\mathbf{x})-d_1(-\Delta)^s\mathcal{Z}(t,\mathbf{x})  \big] d\mathbf{x}dt .
	 \end{equation}
 Let $p>1$. By using H\"{o}lder's inequality and the estimate \eqref{DualityInequality},  we obtain
  \begin{equation}
 	\iint_{Q_T} u_2(t,\mathbf{x}) \varphi(t,\mathbf{x}) d\mathbf{x}dt \leq C\big[ \|u_{01}-u_{02}\|_{L^p(\Omega)} \|\mathcal{Z}_0\|_{L^{p^\prime}(\Omega)} +  (1+d_1) 
 	\|u_1\|_{L^p(Q_T)} 	\|\varphi\|_{L^{p^\prime}(Q_T)}\big],
 \end{equation}
where $p^\prime=\dfrac{p}{p-1}$ and $C:=C(p,T,\Omega)>0$. As $\varphi$ is regular, we deduce by duality that
$$
\|u_2\|_{L^p\left(Q_T\right)} \leq C\left(\left\|u_{01}+u_{02}\right\|_{L^p(\Omega)}+\left(1+d_1\right)\|u_1\|_{L^p\left(Q_T\right)}\right).
$$
Hence,
$$
\|u_2\|_{L^p\left(Q_T\right)} \leq C \max \left\{\left\|u_{01}+u_{02}\right\|_{L^p(\Omega)},1+d_1\right\}\left(1+\|u_1\|_{L^p\left(Q_T\right)}\right) \text {, }
$$
which completes the proof.
\end{proof}

	
	\subsection{ First result of global existence : case of reversible chemical reaction}\label{reversible-chemical-reaction}
%
		Introducing the fractional counterpart of System~\eqref{SystemChemReac-DiffCL}, namely:

		\begin{equation}\label{SystemChemReac-DiffFR} 
			\left\{
			\begin{array}{rclll}\partial_t u_1(t,\textbf{x})+d_1 (-\Delta)^s u_1(t,\textbf{x})&=&	f_1(u_1(t,\textbf{x}),u_2(t,\textbf{x}),u_3(t,\textbf{x}))	, & (t,\textbf{x})\in Q_T, \\ \partial_t u_2(t,\textbf{x})+d_2 (-\Delta)^s u_2(t,\textbf{x})&=&	f_2(u_1(t,\textbf{x}),u_2(t,\textbf{x}),u_3(t,\textbf{x}))	, & (t,\textbf{x})\in Q_T, \\ \partial_t u_3(t,\textbf{x})+d_3 (-\Delta)^s u_3(t,\textbf{x})&=&	f_3(u_1(t,\textbf{x}),u_2(t,\textbf{x}),u_3(t,\textbf{x}))	, & (t,\textbf{x})\in Q_T, \\ u_1(t,\textbf{x})=u_2(t,\textbf{x})=u_3(t,\textbf{x})&=&0,&(t,\textbf{x})\in (0,T)\times(\R^N\setminus\Omega), \\ u_1(0,\textbf{x})&=&u_{01}(\textbf{x}), & \textbf{x} \in \Omega, \\ u_2(0, \textbf{x})&=&u_{02}(\textbf{x}), & \textbf{x} \in \Omega, \\ u_3(0, \textbf{x})&=&u_{03}(\textbf{x}),  & \textbf{x} \in \Omega,\end{array}\right.
		\end{equation}

\noindent where $0<s<1$, $f_1= \alpha_1g,\; f_2= \alpha_2g,\; f_3= -\alpha_3g \text{ with } g=u_3^{\alpha_3}-u_1^{\alpha_1}u_2^{\alpha_2} \text{ and } \alpha_1, \alpha_2, \alpha_3 \geq 1$. 
		
		%
%
\medbreak

\noindent Our {first} result of global existence reads as

\begin{thm}  \label{MainTh1}
Assume that
			$(u_{01},u_{02},u_{03}) \in (L^\infty(\Omega)^+)^3$. System~\eqref{SystemChemReac-DiffFR} admits a unique nonnegative global strong solution in the following cases: 
			\\
			{\em (i)}  $\alpha_1+\alpha_2<\alpha_3$ ;\\
			{\em (ii)} $ \alpha_3=1$ whatever are  $\alpha_1$ and $\alpha_2$ ; \\
			{\em (iii)} $d_1=d_3$ or $d_2=d_3$ whatever are $\alpha_1$,$\alpha_2$ and $\alpha_3$ ;	\\		
		{\em (iv)} $d_1=d_2\neq d_3$  for any $(\alpha_1,\alpha_2,\alpha_3)\in [1,+\infty)^2\times (1,+\infty)$ such that $\alpha_1+\alpha_2\neq\alpha_3$.			
	\end{thm}

Before proceeding to the proof of the previous theorem, it should be mentioned that it is similar to those of \cite[Theorems 1 and 3]{LaamActa2011}.

		\begin{proof}
			Let $T\in(0,T_{\max})$ and $t\in (0,T]$. Moreover, let us denote by $C$ a positive constant which may be different from line to another.

			\medbreak
			\noindent{\bf (i) The case $\alpha_1+\alpha_2<\alpha_3$.}\\
			Let us begin by noting  that $f_3$ and $f_1$ (resp. $f_2$) satisfy {\bf (M)}. In fact, $\alpha_3 f_1+ \alpha_1 f_3=0$ and $\alpha_3 f_2+ \alpha_2 f_3=0$. Therefore, we  deduce from Lemma \ref{FractionalDualityThm} that for any $ 1<p<+\infty$
					\begin{equation}\label{u3-controle-u1}
				\|u_1\|_{L^p\left(Q_T\right)} \leq C\left(1+\|u_3\|_{L^p\left(Q_T\right)}\right)
				\end{equation}
		and
		\begin{equation}\label{u3-controle-u2}
		\|u_2\|_{L^p\left(Q_T\right)} \leq C\left(1+\|u_3\|_{L^p\left(Q_T\right)}\right).
		\end{equation}
		Therefore, in order to prove the global existence, it is sufficient to prove that $u_3\in L^p(Q_{T})$ for $p$ large enough. \\
			Let $q>1$. By multiplying the third equation of System~\eqref{SystemChemReac-DiffFR}  by $(u_3(t,\mathbf{x}))^{q}$ and integrating over $Q_T$, we get
			\begin{equation}\label{xuq+intQT}
				\left\{ \begin{array}{ll} \dis
					\frac{1}{q+1} \int_{\Omega} (u_3(T,\mathbf{x}))^{q+1}d\mathbf{x}+ d_3 \iint_{Q_T} \big[(-\Delta)^s u_3 (t,\mathbf{x}) \big] (u_3(t,\mathbf{x}))^{q} d\mathbf{x}dt+\alpha_3\iint_{Q_T} (w(t,\mathbf{x}))^{q+\gamma} d\mathbf{x}dt
					\\\\\dis=\alpha_3\iint_{Q_T} (u_1(t,\mathbf{x}))^{\alpha_1} (u_2(t,\mathbf{x}))^{\alpha_2} (u_3(t,\mathbf{x}))^q  d\mathbf{x}dt+\frac{1}{q+1} \int_{\Omega} (u_{03}(\mathbf{x}))^{q+1} d\mathbf{x}.
				\end{array}\right.
			\end{equation}
			\\
			By using the the so-called Stroock-Varopoulos inequality (see, for instance, \cite[Formula (B7)]{BonFigRosOt}), we have
			\begin{equation}
				\iint_{Q_T} \big[(-\Delta)^s u_3 (t,\mathbf{x}) \big] (u_3(t,\mathbf{x}))^{q} d\mathbf{x}dt \geq \frac{4q}{(q+1)^2}   \iint_{Q_T} \left| (-\Delta)^{\frac{s}{2}} (u_3(t,\mathbf{x}))^{\frac{q+1}{2}}  \right|^2 d\mathbf{x}dt\geq 0.
			\end{equation}
			Furthermore, by using H\"{o}lder's inequality, we have
			\begin{equation}
				\iint_{Q_T} (u_1(t,\mathbf{x}))^{\alpha_1} (u_2(t,\mathbf{x}))^{\alpha_2} (u_3(t,\mathbf{x}))^q  d\mathbf{x}dt \leq
				\|u_1\|_{L^{\alpha_1 \beta}\left(Q_T\right)}^{\alpha_1}\|u_2\|_{L^{\alpha_2 \gamma}\left(Q_T\right)}^{\alpha_2}\|u_3\|_{L^{\alpha_3+q}\left(Q_T\right)}^q,
			\end{equation}
			where
			$$
			\frac{1}{\beta}+\frac{1}{\gamma}+\frac{q}{q+\alpha_3}=1.
			$$
			Since  $\alpha_1+\alpha_2<\alpha_3$,  we can choose $\beta$ such that $\beta \alpha_1 \leq q+\alpha_3$  and  $\gamma$ such that  $\gamma \alpha_2 \leq q+\alpha_3$. Then, $L^{q+\alpha_3}\left(Q_T\right) \subset L^{\alpha_1 \beta}\left(Q_T\right)$   and  $L^{q+\alpha_3}\left(Q_T\right) \subset L^{\alpha_2 \gamma}\left(Q_T\right)$.    
			Consequently, we have
			\begin{equation}\label{Linclusion}
				\iint_{Q_T} (u_1(t,\mathbf{x}))^{\alpha_1} (u_2(t,\mathbf{x}))^{\alpha_2} (u_3(t,\mathbf{x}))^q  d\mathbf{x}dt \leq C
				\|u_1\|_{L^{\alpha_3+q}\left(Q_T\right)}^{\alpha_1}\|u_2\|_{L^{\alpha_3+q}\left(Q_T\right)}^{\alpha_2}\|u_3\|_{L^{\alpha_3+q}\left(Q_T\right)}^q.
			\end{equation}
			By using the same concept of duality as of Lemma \ref{FractionalDualityThm}, we obtain
			\begin{equation}\label{u1norm}
				\|u_1\|_{L^{\alpha_3+q}\left(Q_T\right)} \leq C\left(1+\|u_3\|_{L^{\alpha_3+q}\left(Q_T\right)}\right)
			\end{equation}
			and 
			\begin{equation} \label{u2norm}
				\|u_2\|_{L^{\alpha_3+q}\left(Q_T\right)} \leq C\left(1+\|u_3\|_{L^{\alpha_3+q}\left(Q_T\right)}\right) .
			\end{equation}
			Thanks to \eqref{u1norm} and \eqref{u2norm}, Formula \eqref{Linclusion} implies that
			\begin{equation}\left\{
				\begin{array}{ll}
					\dis
					\iint_{Q_T} (u_1(t,\mathbf{x}))^{\alpha_1} (u_2(t,\mathbf{x}))^{\alpha_2} (u_3(t,\mathbf{x}))^q  d\mathbf{x}dt\\\\\dis\leq C\left(1+\|u_3\|_{L^{\alpha_3+q}\left(Q_T\right)}\right)^{\alpha_1}\left(1+\|u_3\|_{L^{\alpha_3+q}\left(Q_T\right)}\right)^{\alpha_2}\left(1+\|u_3\|_{L^{\alpha_3+q}\left(Q_T\right.}\right)^q .
				\end{array}
				\right.
			\end{equation}
			If $\|u_3\|_{L^{\alpha_3+q}\left(Q_T\right)} \leq 1$ then the proof ends up. Otherwise, 
			\begin{equation}
				\iint_{Q_T} (u_1(t,\mathbf{x}))^{\alpha_1} (u_2(t,\mathbf{x}))^{\alpha_2} (u_3(t,\mathbf{x}))^q  d\mathbf{x}dt \leq C\|u_3\|_{L^{\alpha_3+q}\left(Q_T\right)}^{q+\alpha_1+\alpha_2}.
			\end{equation}
			Hence, we have by \eqref{xuq+intQT} 
			\begin{equation}
				\iint_{Q_T} (u_3(t,\mathbf{x}))^{q+\alpha_3}\leq C\|u_3\|_{L^{\alpha_3+q}\left(Q_T\right)}^{q+\alpha_1+\alpha_2}+\frac{1}{q+1} \int_{\Omega} (u_{03}(\mathbf{x}))^{q+1} d\mathbf{x}.
			\end{equation}
			Therefore, by using Young's inequality, we get 
			\begin{equation}
				\|u_3\|_{L^{\alpha_3+q}\left(Q_T\right)}\leq C.
			\end{equation}
			By choosing $\dfrac{\alpha_3+q}{\alpha_3}>\dfrac{N+2s}{2s}$, we have by Theorem \ref{BoundedSol} 
			\begin{equation}\label{u1bounded}
				\|u_1\|_{L^\infty(Q_T)}\leq C
			\end{equation}
		and 	\begin{equation}\label{u2bounded}
			\|u_2\|_{L^\infty(Q_T)}\leq C.
		\end{equation}
			Now going back to the third equation of~\eqref{SystemChemReac-DiffFR}, we deduce from \eqref{u1bounded} and \eqref{u2bounded} that
			\begin{equation}
				\|u_3\|_{L^{\infty}\left(Q_T\right)}\leq C,
			\end{equation}
			which implies that $T_{\max}=+\infty$.
		\medbreak
		\noindent {\bf (ii) The case $\alpha_3=1$. }
		From the semigroup property and thanks to comparison principle, it holds for any $p>1$
			\begin{equation}\label{inequ_1}
				\|u_1(t,\cdot)\|_{L^p(\Omega)} \leq\left\|u_0\right\|_{L^p(\Omega)}+\int_0^t\|u_3(\tau,\cdot)\|_{L^p(\Omega)} d\tau.
			\end{equation}
			By applying H\"{o}lder's inequality and the same duality arguments as \eqref{u3-controle-u1}, we obtain
			\begin{equation}
				\int_0^t\|u_3(\tau,\cdot)\|_{L^p(\Omega)} d\tau \leq C t^{1 / {p^\prime}}\left[1+\left(\int_0^t \int_{\Omega} (u_1(\tau,\mathbf{x}))^p d\tau d \mathbf{x}\right)^{1 / p}\right],
			\end{equation}
			where $p^{\prime}=\dfrac{p}{p-1}$.
			For any $t \in(0, T]$, let us set $h(t):=\dis\int_{\Omega}|u_1(t, \mathbf{x})|^p d\mathbf{x}$. Then, \eqref{inequ_1} implies that
			\begin{equation}\label{ineqh}
				h(t)^{\frac{1}{p}} \leq C\left[1+\left(\int_0^t h(\tau) d\tau\right)^{\frac{1}{p}}\right].
			\end{equation}
			By taking the $p$th power of \eqref{ineqh}, we obtain
			\begin{equation}
				h(t) \leq C\left[1+\int_0^t h(\tau) d\tau \right].
			\end{equation}
			Therefore, by applying Gronwall's Lemma, we get
			\begin{equation}\label{estu1}
				\|u_1\|_{L^p\left(Q_T\right)}\leq C .
			\end{equation}
			Repeating the method above with $u_2$ instead of $u_1$, we obtain
			\begin{equation} \label{estu2}
				\|u_2\|_{L^p\left(Q_T\right)}\leq C.
			\end{equation}
			Estimates \eqref{estu1} and \eqref{estu2} imply that for some $q>\frac{N+2s}{2s}$
			\begin{equation}
				\|u_1^{\alpha_1} u_2^{\alpha_2}\|_{L^q\left(Q_T\right)} \leq C.
			\end{equation}
			Going back to the third equation of System \eqref{SystemChemReac-DiffFR}, we have by Theorem \ref{BoundedSol}
			\begin{equation}
				\|u_3\|_{L^{\infty}\left(Q_T\right)}\leq C.
			\end{equation}
			Hence, $T_{\max}=+\infty$ which completes the proof in this case.
			
			\medbreak
			
			\noindent {\bf (iii)} and  {\bf (iv).} The proofs in these two cases are  straightforward and left to the reader.
		\end{proof}
		%
		%

\subsection{Second result of global existence : case of triangular structure}\label{TriangularStructure}
	
Our second  result of global existence reads as
		\begin{thm}\label{MainTh2} 
		Let
		$(u_{01},\ldots,u_{0m}) \in (L^\infty(\Omega)^+)^m$ where $m>2$. Besides {\bf (P)},   assume that $f_i$ is  
			 at most polynomial for each $i=1,\ldots,m$. Moreover, assume that there exist a vector
			$\mathbf{b}\in[0,+\infty)^m$, %
			and a lower triangular invertible matrix $Q\in {\mathcal{M}_m([0,+\infty))}$,
			such that
			\begin{equation}\label{TriangStructure}
			\forall\mathbf{r}=(r_1,\ldots,r_m)\in [0,+\infty)^m,\; Q\mathbf{f}(\mathbf{r})\leq \Big[1+\dis\sum\limits_{i=1}^m r_i\Big]\mathbf{b}.
		\end{equation}
			Then, System \eqref{ReactiondiffusionSystem} admits a unique nonnegative global strong  solution.
		\end{thm}

	
		\begin{proof}
			Let us denote
			$Q=(q_{ij})_{i,j=1,\ldots,m}$,	$\mathbf{b}=(b_1,\ldots,b_m)^T$ and $\mathbf{f}(\mathbf{r})=(f_1(\mathbf{r}),\ldots,f_m(\mathbf{r}))^T$.
		\\
		 By Lemma \ref{LocalExistence}, System \eqref{ReactiondiffusionSystem} admits a unique strong solution $\mathbf{u}$ on $[0,T_{\max})\times \Omega$.  In what follows, we prove that this solution is global, {\it i.e.} $T_{\max}=+\infty$.
			Let us denote by $C$ a positive constant which may be different from line to another. Moreover, let $T\in(0,T_{\max})$ and $t\in (0,T]$.\\
			%
			\noindent	For each $i=1,\ldots,m$, we denote by $w_i$ the weak solution to
			\begin{equation}
				\label{Testfunc}
				\left\{  
				\begin{array}{rcll}
					\partial_t w_i(t,\mathbf{x})+d_i (-\Delta)^s w_i(t,\mathbf{x}) &=& \Big[1+\dis\sum\limits_{i=1}^m u_i(t,\mathbf{x})\Big] \dfrac{b_i}{q_{ii}}, &(t,\mathbf{x})\in Q_T,\\
					w_i(t,\mathbf{x})&=&0, &(t,\mathbf{x})\in (0,T)\times(\R^N\setminus\Omega),\\
					w_i(0,\mathbf{x})&=&0,& \mathbf{x}\in\Omega.
				\end{array}
				\right.
			\end{equation}	
	Then, $w_i$ fulfills 
		\begin{equation}
			\label{ineqlasttheo5}
			w_i(t,\mathbf{x})=\dis\int_0^t S_i(t-\tau)\Big[1+  \dis\sum\limits_{i=1}^mu_i(\tau,\mathbf{x}) \Big]\dfrac{b_i}{q_{ii}}d\tau,\quad (t,\mathbf{x})\in [0,T)\times\Omega.
		\end{equation}
	According to \eqref{TriangStructure}, we have 
	\begin{equation}\label{sum-}
		\sum\limits_{j=1}^i q_{ij} f_i(u_1(t,\textbf{x}),\ldots,u_m(t,\textbf{x}))\leq \Big[ 1+    \sum\limits_{j=1}^m  u_j (t,\textbf{x})  \Big] b_i.
	\end{equation}
%
		%
		$\bullet$	The case $i=1$. We subtract the first equation  of \eqref{Testfunc} from the first equation of \eqref{ReactiondiffusionSystem} multiplied by $q_{11}$. Then,
			\begin{equation}
			q_{11}\Big[\partial_t(u_1-w_1)(t,\textbf{x})+d_1(-\Delta)^s(u_1-w_1)(t,\textbf{x})  \Big]\leq 0.
		\end{equation}	
	Then, by Theorem \ref{maximumprinciple}, we get
	$$
	\|u_1-w_1\|_{L^\infty(Q_{T})}\leq \|u_{01}\|_{L^\infty(\Omega)}.
	$$
		Therefore, 
	$$
\|u_1\|_{L^\infty(Q_{T})} \leq  C [ 1+ \|	w_1 \|_{L^\infty(Q_{T})} ].
	$$
	By taking the $L^\infty(\Omega)$-norm of \eqref{ineqlasttheo5} and using the $L^\infty$-contractivity of $\{S_A(t)\}_{t\geq0}$, we get
		\begin{equation}
\|u_1(t,\cdot)\|_{L^\infty(\Omega)} \leq  C \Big[ 1+ \int_0^t\Big\| \sum\limits_{i=1}^m 	u_i(\tau,\cdot) \Big\|_{L^\infty(\Omega)} d\tau \Big] .
		\end{equation}		
	Hence,
			\begin{equation}\label{u_1}
	\|u_1(t,\cdot)\|_{L^\infty(\Omega)} \leq  C \Big[ 1+ \int_0^t \|u_1(\tau,\cdot)\|_{L^\infty(\Omega)} d\tau+ \int_0^t\Big\| \sum\limits_{i=2}^m 	u_i(\tau,\cdot) \Big\|_{L^\infty(\Omega)} d\tau \Big] .
		\end{equation}		
	%
%
			\\
			$\bullet$ The case $i\geq 2$.
		{By subtracting the $i$th equation of \eqref{Testfunc} from the $i$th equation of \eqref{ReactiondiffusionSystem} multiplied by $q_{ii}$, we obtain}
	%
			\begin{equation}
				\label{ineqlasttheo1}
				q_{ii}\Big[\partial_t(u_i-w_i)(t,\textbf{x})+d_i(-\Delta)^s(u_i-w_i)(t,\textbf{x})  \Big]\leq -\dis\sum\limits_{j=1}^{i-1} q_{ij}\Big[ \partial_t u_j (t,\textbf{x}) +d_j(-\Delta)^s u_j(t,\textbf{x})  \Big].
			\end{equation}
		Now, we use the same concept of duality as of Lemma \ref{FractionalDualityThm}. More precisely,
			let $\varphi$ be a nonnegative regular function and let $\mathcal{Z}$ be the solution to Problem \eqref{DualProblem} with $d=\frac{d_i}{q_{ii}}$. Then,  we multiply the two sides of the inequality \eqref{ineqlasttheo1} by $\mathcal{Z}$ and integrate over $Q_t$. 
			We obtain
			\begin{equation}\label{I1I2} \left\{\begin{array}{ll}
		\dis\iint_{Q_t}q_{ii}\partial_t (u_i-w_i) (\tau,\textbf{x}) \mathcal{Z}(\tau,\textbf{x})d\tau d\textbf{x}+\iint_{Q_t}d_i(-\Delta)^s(u_i-w_i)(\tau,\textbf{x}) \mathcal{Z}(\tau,\textbf{x})d\tau d\textbf{x}:=I_1
	\\\\\leq-\dis \sum\limits_{j=1 }^{i-1} q_{ij}
	\iint_{Q_t}\dis	\partial_t u_j (\tau,\textbf{x})\mathcal{Z}(\tau,\textbf{x})d\tau d\textbf{x}-\sum\limits_{j=1 }^{i-1} q_{ij}d_j\iint_{Q_t}(-\Delta)^su_j(\tau,\textbf{x})\mathcal{Z}(\tau,\textbf{x})d\tau d\textbf{x}:=I_2.
				\end{array}\right.
			\end{equation}
			Now, let us simplify the terms $I_1$ and $I_2$.\\
		We have
		\begin{equation}
			\begin{array}{ll}   	
				I_1
				\displaystyle 
				= q_{ii} \iint_{Q_t}(u_i-w_i)(\tau,\textbf{x})\varphi(\tau,\textbf{x}) d\tau d\textbf{x} - q_{ii}\int_{\Omega} u_{0i}(\textbf{x}) \mathcal{Z}_0 (\textbf{x})d\textbf{x}.
			\end{array}
		\end{equation}
		In addition, we have		
	\begin{equation}
		\begin{array}{ll}   	
		\dis	I_2 = \sum\limits_{j=1}^{i-1} q_{ij} \iint_{Q_t}u_j(\tau,\textbf{x})\left( \partial_t\mathcal{Z}(\tau,\textbf{x})-d_j(-\Delta)^s\mathcal{Z}(\tau,\textbf{x}) \right)d\tau d\textbf{x} +\sum\limits_{j=1}^{i-1} q_{ij}\int_{\Omega} u_{0j} (\textbf{x})\mathcal{Z}_{0}(\textbf{x}) d\textbf{x}.
		\end{array}
	\end{equation}		
			Then, \eqref{I1I2} implies
		\begin{equation}
			\left\{
			\begin{array}{ll}	
			\displaystyle	q_{ii} \iint_{Q_t}(u_i-w_i)(\tau,\textbf{x})\varphi(\tau,\textbf{x}) d\tau d\textbf{x}\\\leq \dis \sum\limits_{j=1 }^i q_{ij}\int_{\Omega} u_{0j} (\textbf{x})\mathcal{Z}_{0}(\textbf{x}) d\textbf{x}+\sum\limits_{j=1}^{i-1} q_{ij} \iint_{Q_t}u_j(\tau,\textbf{x})\left( \partial_t\mathcal{Z}(\tau,\textbf{x})-d_j(-\Delta)^s\mathcal{Z}(\tau,\textbf{x}) \right)d\tau d\textbf{x}.
			\end{array} 
		\right.
		\end{equation}
{Let $p>1$ large enough and $p^\prime=\dfrac{p}{p-1}$.} By using H\"{o}lder's inequality and the estimate \eqref{DualityInequality},  we get
   	\begin{equation}
   	\left\{
   	\begin{array}{ll}
   		\displaystyle	\iint_{Q_t}(u_i-w_i)(\tau,\textbf{x})\varphi(\tau,\textbf{x}) d\tau d\textbf{x}\\\\\displaystyle\leq C  \dis \sum\limits_{j=1}^i \|u_{0j}\|_{L^p(\Omega)}\|\mathcal{Z}_0\|_{L^{p^\prime}(\Omega)}+C  \dis \sum\limits_{j=1}^{i-1}\|u_j\|_{L^p(Q_t)}\left(\left\| \partial_t\mathcal{Z} \right\|_{L^{p^\prime}(Q_t)} + d_j \left\| (-\Delta)^s\mathcal{Z} \right\|_{L^{p^\prime}(Q_t)}  \right)\\\\
   		\displaystyle\leq \displaystyle C\dis \sum\limits_{j=1 }^i\|u_{0j}\|_{L^p(\Omega)} \|\varphi\|_{L^{p^\prime}(Q_t)} +C \dis \sum\limits_{j=1 }^{i-1} \|u_j\|_{L^p(Q_t)}\|\varphi\|_{L^{p^\prime}(Q_t)}\\\\
   		\displaystyle\leq C\Big(1+\dis \sum\limits_{j=1}^{i-1} \|u_j\|_{L^p(Q_t)}\Big)\|\varphi\|_{L^{p^\prime}(Q_t)}.
   	\end{array}
   	\right.
   \end{equation}		
	As $\varphi$ is regular, we obtain by duality that 
			\begin{equation}
				\label{ineqlasttheo2}
				\|u_i-w_i\|_{L^p(Q_t)}\leq C \Big( 1+ \dis\sum\limits_{j=1}^{i-1}\|u_j\|_{L^p(Q_t)}\Big).
			\end{equation}
			Hence, by an elementary induction on $i$, we obtain 
			\begin{equation}
				\label{ineqlasttheo2-}
				\|u_i\|_{L^p(Q_t)}\leq C \Big( 1+ \dis\sum\limits_{j=1 }^i\|w_j\|_{L^p(Q_t)}\Big).
			\end{equation}
			Consequently, by summing over $i$ and taking the $p$th power, we obtain 
			\begin{equation}
				\label{ineqlasttheo3}
				\Big\|\dis\sum\limits_{i=2}^mu_i\Big\|^p_{L^p(Q_t)}\leq C \Big( 1+ \sum\limits_{i=2}^m\|w_i\|^p_{L^p(Q_t)}\Big).
			\end{equation}
		\noindent Now, let us take the $L^p(\Omega)$-norm of \eqref{ineqlasttheo5}.	Then, by using the estimate \eqref{Ultracontractive} with $q=p$ and H\"{o}lder's inequality, we get 
			\begin{equation}
				\label{ineqlasttheo6}
				\|w_i(t,\cdot)\|^p_{L^p(\Omega)}\leq C \Big[ 1+   \Big\|\dis\sum\limits_{i=2}^mu_i\Big\|^p_{L^p(Q_t)}   \Big].
			\end{equation}
			Hence, we obtain by \eqref{ineqlasttheo3}
			\begin{equation}
				\label{ineqlasttheo7}
				\dis\sum\limits_{i=2}^m\|w_i(t,\cdot)\|^p_{L^p(\Omega)}\leq C \left[ 1+  	\dis  \int_0^t\sum\limits_{i=2}^m \|w_i(\tau,\cdot)\|^p_{L^p(\Omega)}   d\tau \right].
			\end{equation}	
		By applying  Gronwall's Lemma, we get $$\dis\sum\limits_{i=2}^m\|w_i(t,\cdot)\|^p_{L^p(\Omega)}<g(t),$$   where $g(t):[0,+\infty)\to [0,+\infty)$ is a non-decreasing and continuous function. Thus, for each $ i=2,\ldots,m$, $w_i(t,\cdot)$ is bounded in $L^\infty(\Omega)$, and so is $\dis\sum\limits_{i=2}^mu_i(t,\cdot)$ thanks to \eqref{ineqlasttheo3}.
			By hypothesis, 
			$f_i$ is at most polynomial. Then, we deduce that $f_i\in L^q(Q_{T})$ {where $q$ is large enough. Then, we can choose
				$q>\dfrac{N+2s}{2s}$} so that we have by Theorem \ref{BoundedSol}:
			$$
			\|u_i\|_{L^\infty(Q_{T})}\leq C \quad \forall i=2,\ldots,m.
			$$
			Going back to \eqref{u_1}, we also apply Gronwall's lemma to get $\|u_1\|_{L^\infty(Q_{T})}\leq C$. 
			Consequently, $T_{\max}=+\infty$.
		\end{proof}

		
		\section{Reaction-Diffusion System with exponential growth : Numerical simulations}\label{S_exp}
		
		The main objective of this section is to examine numerically
		the existence of 
		{strong} solutions to 
		the following system :
		\begin{equation}\label{ExampleReactiondiffusionSystem2x2} \tag{$S_{\exp}$}
			\left\{   \begin{array}{rclll}
				\partial_t u(t,\textbf{x})+d_1(-\Delta)^{s}u(t,\textbf{x})&=&\hskip-2mm-u(t,\textbf{x})e^{(v(t,\textbf{x}))^\beta},&(t,\textbf{x})\in Q_T,\\ 
				\partial_t v(t,\textbf{x})+d_2(-\Delta)^{s}v(t,\textbf{x})&=&\hskip1mm u(t,\textbf{x})e^{(v(t,\textbf{x}))^\beta},&(t,\textbf{x})\in Q_T,\\
				u(t,\textbf{x})=v(t,\textbf{x})&=&0,&(t,\textbf{x})\in (0,T)\times(\R^N\setminus\Omega),\\
				u(0,\textbf{x})&=&u_0(\textbf{x}),	& \textbf{x}\in\Omega,\\
				v(0,\textbf{x})&=&v_0(\textbf{x}),	& \textbf{x}\in\Omega,
			\end{array}    \right.
		\end{equation}
		where $d_1,d_2>0$, $u_0,v_0\in (L^\infty(\Omega))^+$ {in the case where $\beta>1$}.\\
		As stated in the introduction, the global existence of {strong} solutions to \eqref{ExampleReactiondiffusionSystem2x2} remains an open problem even for the classical case ($s=1$).
		{In \cite{LaamMalZit}, the second author and his coworkers have investigated numerically this open problem  in the classical case (see also \cite{Malekthese}).
		In the present paper, we use a different approach adapted to the fractional Laplacian 
		in order to treat System \eqref{ExampleReactiondiffusionSystem2x2}.	}
More precisely, we perform our numerical simulations for 2D-case by
		considering the unit ball
		$\Omega=B(0,1)\subset \R^2 $.
		Furthermore, we carry out these simulations using :
		\begin{itemize}
			\item[---] the  Crank--Nicolson scheme to approximate the time derivative;
			\vspace{-2mm}\item[---] the finite element method (FEM) to discretise in space and approximate the solution. To realize this, we rely on 
			\cite{AcosBersBort}  where the authors proposed a FEM adapted to the fractional 
			case as well as a Matlab code in 2D-case.  Notably, we use this code and modify it to adapt it to our problem;
			\vspace{-2mm}
			\item[---] the fixed-point iteration method to solve the resulting nonlinear equations at each discrete time. 
		\end{itemize}
		\vspace{-2mm}
		As in \cite{AcosBersBort}, we consider  an admissible triangulation of $\overline{\Omega}$ and we denote by $h$ its mesh size.
		Regarding the temporal discretization, let  $
		0=t_0<t_1<t_2<\ldots<t_{\tilde{N}}=T
		$ be a partition of the interval $(0,T)$ with time step $h_t=t_{n+1}-t_{n} $.
		As for the fixed-point iterative processes, they are performed under the control of a tolerance level equal to $10^{-6}$.

		\smallbreak\noindent
		Before going further, we would like to point out that we use numerical simulations to build a conjecture about the existence of solutions to System \eqref{ExampleReactiondiffusionSystem2x2}, without conducting any theoretical study of errors. That being said, it should be recalled that the Crank-Nicolson scheme is a second-order numerical method
%
		{(see, for instance, \cite[Chapter 2]{MortonMayers})}. 
		Moreover, according to \cite{AcosBersBort}, the order of convergence of the FEM in the linear case (Poisson equation for example) depends on the value of $s$.  Namely, for $0<s\leq \frac12$, the expected order of convergence with respect to $h$ is $s+\frac{1}{2}$, while for $\frac12<s<1$, the expected order of convergence is 1. On the other hand, let us  comment on the convergence of the fixed-point iterations. 
	Although	 its theoretical analysis is limited in our case, we are able to numerically estimate the constants of contraction at each iteration and verify that they remain $<1$.
	Furthermore, it should be stressed tha we present here  a partial numerical investigation of System \eqref{ExampleReactiondiffusionSystem2x2}. For more details, we refer the interested reader to \cite{DaouThese}.
	 }
		
		\medbreak
		Now, let us present some numerical simulations that illustrate the performance of our numerical scheme.
		To do so,
		let us consider the function $\rho(\textbf{x})= (1-\|\textbf{x}\|^2)^s$. It is known that $(-\Delta)^s \rho(\textbf{x})= \lambda(s)$, where $\lambda(s):=4^s \Gamma(s+1)^2$ (see \cite{dyda}). We use this fact to construct a solution to System \eqref{ExampleReactiondiffusionSystem2x2} with additional terms on the right-hand sides (r.h.s.). Specifically, the couple $(u(t,\textbf{x}),v(t,\textbf{x}))=(1,1/2) \rho(\textbf{x}) e^{-t}$ solves the system with the following r.h.s. terms:
		
		{\small $$
			\left\{
			\begin{array}{llll}
				\text{r.h.s.1} &=& -u(t,\textbf{x})e^{v(t,\textbf{x})^\beta} + \left( d_1 \lambda(s) - \rho(\textbf{x})+ \rho(\textbf{x}) \exp\left[2^{-\beta} e^{-\beta t} \rho(\textbf{x})^{\beta} \right] \right) e^{-t}\\
				\text{r.h.s.2} &=&\hskip2.5mm u(t,\textbf{x})e^{v(t,\textbf{x})^\beta} + \left(1/2 d_2 \lambda(s) - 1/2 \rho(\textbf{x})- \rho(\textbf{x}) \exp\left[2^{-\beta} e^{-\beta t} \rho(\textbf{x})^{\beta} \right] \right) e^{-t}.
			\end{array}
			\right.
			$$}
		
		\noindent	To assess the accuracy of our scheme, we compare the approximated solution (denoted with the subscript ``$h$'') to the exact one (denoted with the subscript  ``$ex$''). Furthermore, we denote by $U_h$ and $V_h$ the last vectors obtained at a certain time after performing the necessary fixed-point iterations. We also denote by $U_{ex}$ and $V_{ex}$ the vectors which contain the exact solution at the same nodes as $U_h$ and $V_h$ respectively. 
		In what follows, we use these notations.\\
		Moreover, we consider different values of $s$, $d_1$, $d_2$ and $\beta$. In fact, for  $s=0.25$ we fix $(d_1,d_2,\beta)=(1,3,2)$, for  $s=0.5$ we fix $(d_1,d_2,\beta)=(2,1,3)$, for $s=0.75$ we fix $(d_1,d_2,\beta)=(1,2, 3)$ and for $s=0.9$ we fix $(d_1,d_2,\beta)=(3,4,2)$.
		
		\noindent
		In {\bf Figure \ref{figure1}}, we have plotted the evolution in 1D of $U_h$ and $V_h$ for $s=0.75$ over time for $t=1,2,3,4,5$ by fixing the $y$-axis at 0.
		
		\noindent
		{\bf Figures \ref{figure2}} and {\bf\ref{figure3}} show the comparison graphs between the numerical and exact solutions for $s=0.75$ at $t=5$.

		\begin{center}
			\begin{minipage}{\linewidth}
				\centering
				\includegraphics[scale=0.5]{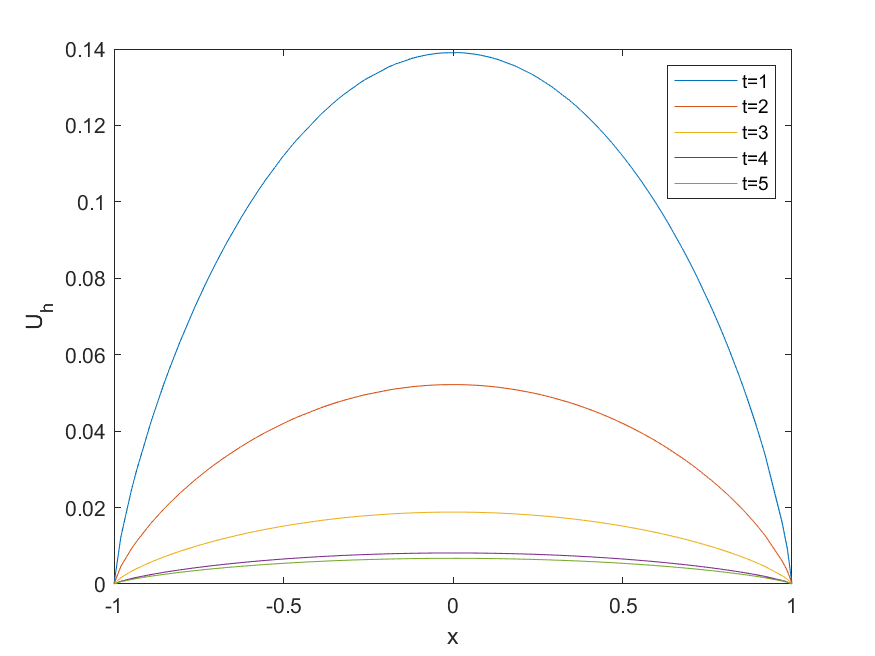} 	\includegraphics[scale=0.5]{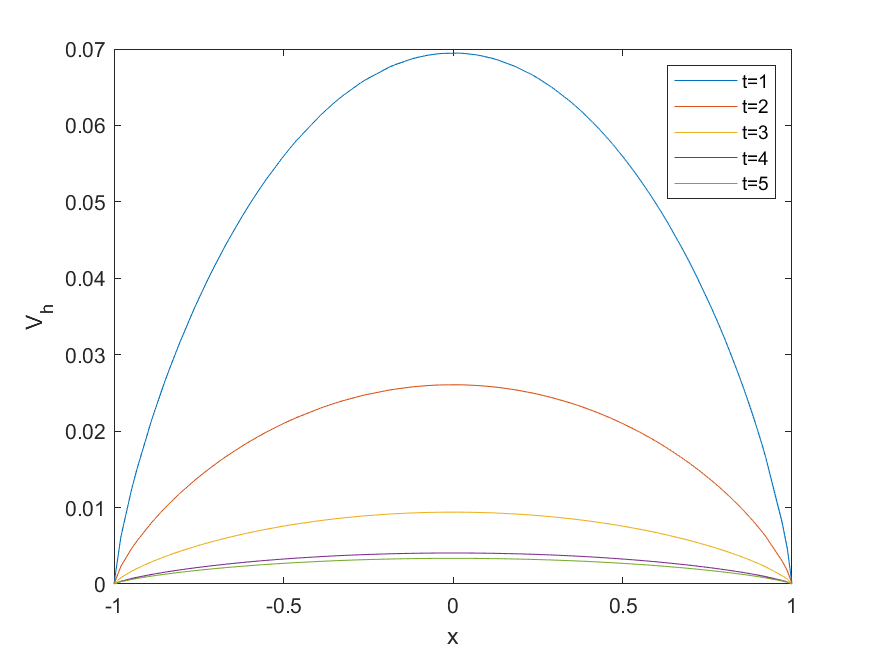}
				\captionof{figure}{Evolution in time of $U_h$ and $V_h$ for $s=0.75$ and $h=\frac{1}{48}$ with the $y$-axis fixed at 0}
				\label{figure1}
			\end{minipage}
		\end{center}

		\begin{center}
			\begin{minipage}{\linewidth}
				\centering
				\includegraphics[scale=0.5]{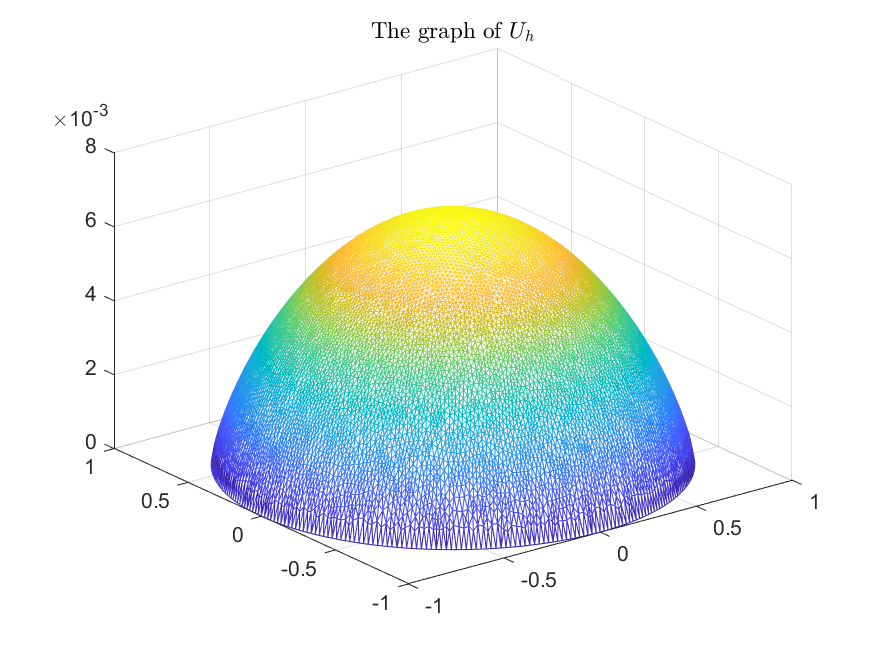}  	\includegraphics[scale=0.5]{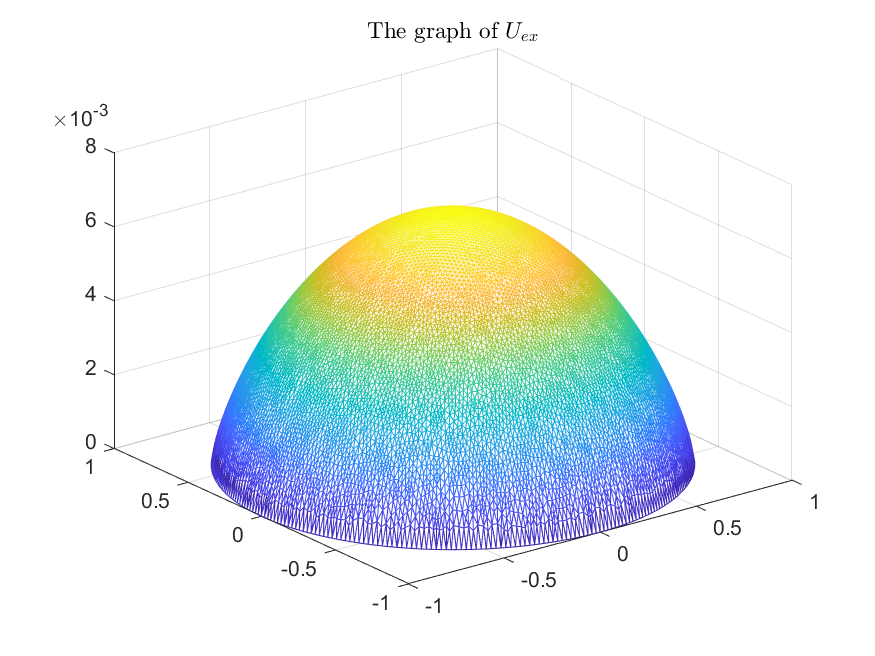}
				\captionof{figure}{	Comparison between $U_h$ and $U_{ex}$ for $s=0.75$ and $h=\frac{1}{48}$}
				\label{figure2}
			\end{minipage}
		\end{center}

		\begin{center}
			\begin{minipage}{\linewidth}
				\centering
				\includegraphics[scale=0.5]{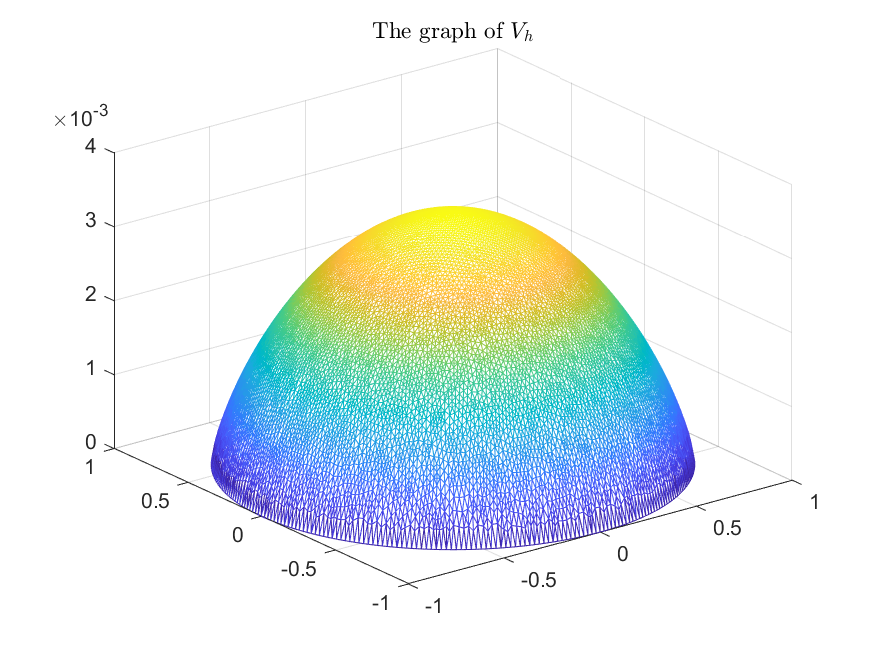}  	\includegraphics[scale=0.5]{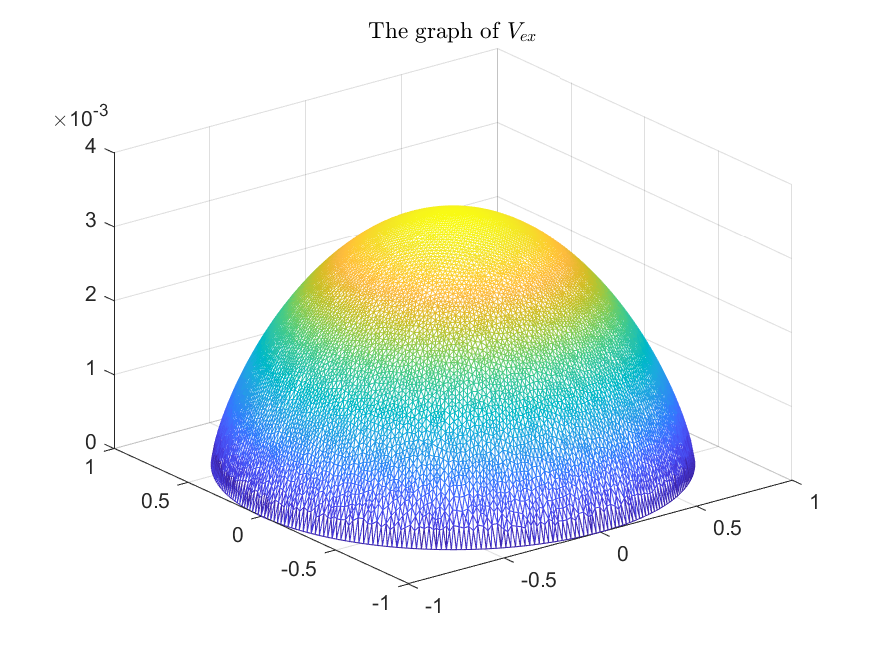}
				\captionof{figure}{	Comparison between $V_h$ and $V_{ex}$ for $s=0.75$ and $h=\frac{1}{48}$ at $t=5$ at $t=5$}
				\label{figure3}
			\end{minipage}
		\end{center}
		
		In order to verify the order of convergence of the FEM in the presented scheme,  we have plotted in {\bf Figure \ref{figure4}} the evolution of the  $L^2$-error in $\Omega$ between the numerical and exact solutions for different values of $s$.  
		
		\begin{center}
			\begin{minipage}{\linewidth}
				\centering
				\includegraphics[scale=0.49]{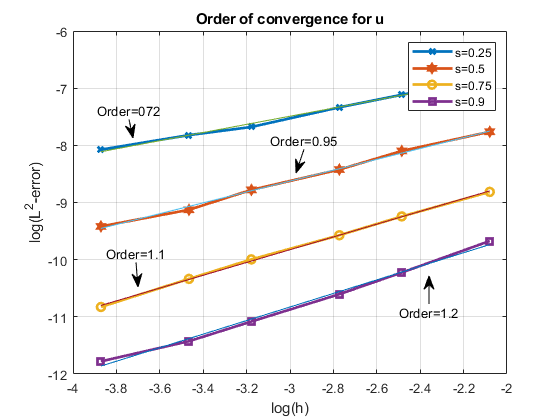} \includegraphics[scale=0.49]{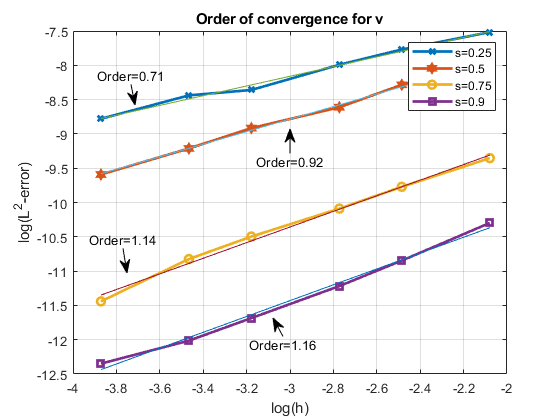}
				\captionof{figure}{Order of convergence for $h=\frac18,\frac{1}{12},\frac{1}{16},\frac{1}{24},\frac{1}{32},\frac{1}{48}$}
				\label{figure4}
			\end{minipage}
		\end{center}
		
		As shown in {\bf Figure \ref{figure4}}, our computed convergence orders for different values of $s$ match the theoretical expectations reported in \cite{AcosBersBort}.
		
		\medbreak
		
		Now, after testing the performance of our scheme, let us go back to our System \eqref{ExampleReactiondiffusionSystem2x2} with  $u_0(\textbf{x})=\rho(\mathbf{x})$ and  $v_0(\textbf{x})=1/2 \rho(\mathbf{x})$.  Our main focus 
		is to investigate the existence of a nonnegative numerical solution to  System \eqref{ExampleReactiondiffusionSystem2x2}  that may be globally time. To this end, we aim to numerically solve the system until a sufficiently large final time   ($t= 5\times 10^{10}$ in our case). 
		
		\noindent
		{\bf Figures \ref{figure5}} and {\bf\ref{figure6}} show the graphs of the numerical solutions obtained at $t=5\times 10^{10}$ for $s=0.75$ and different values of $d_1$, $d_2$ and $\beta$  by fixing $h_t=10^{-2}$ and $h=\frac{1}{48}$.

		\begin{center}
			\begin{minipage}{\linewidth}
				\centering
				\includegraphics[scale=0.5]{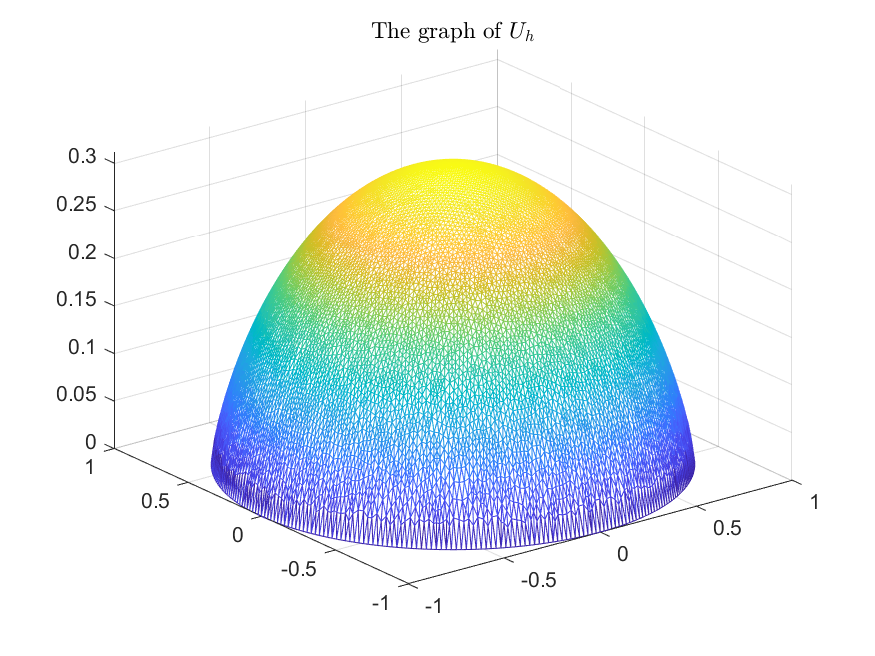} \includegraphics[scale=0.5]{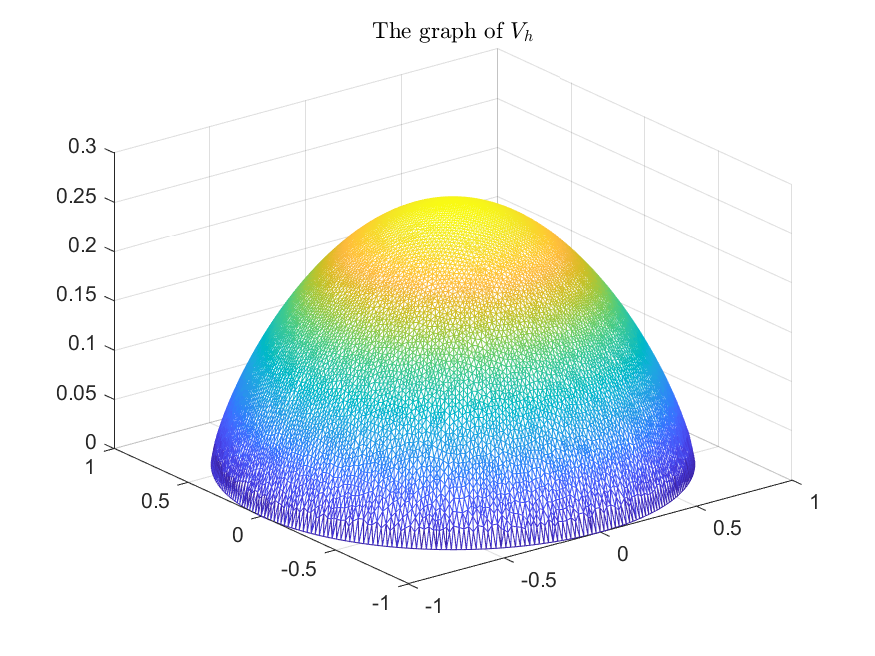}
				\captionof{figure}{	Plots of $U_h$ and $V_{h}$ for $s=0.75$ and $(d_1,d_2,\beta)=(1,2,3)$ at $t=5\times 10^{10}$}
				\label{figure5}
			\end{minipage}
		\end{center}

		\begin{center}
			\begin{minipage}{\linewidth}
				\centering
				\includegraphics[scale=0.5]{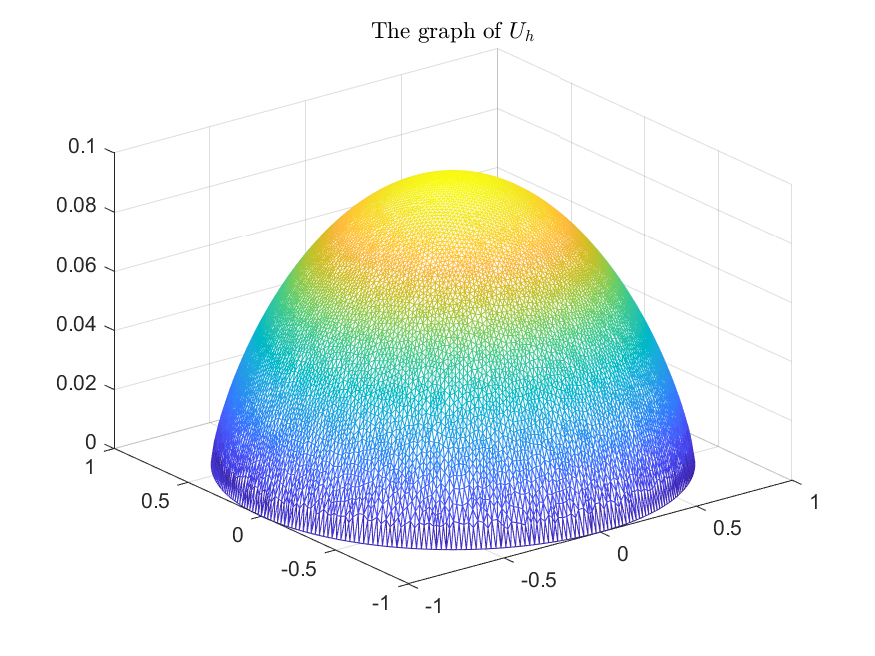} \includegraphics[scale=0.5]{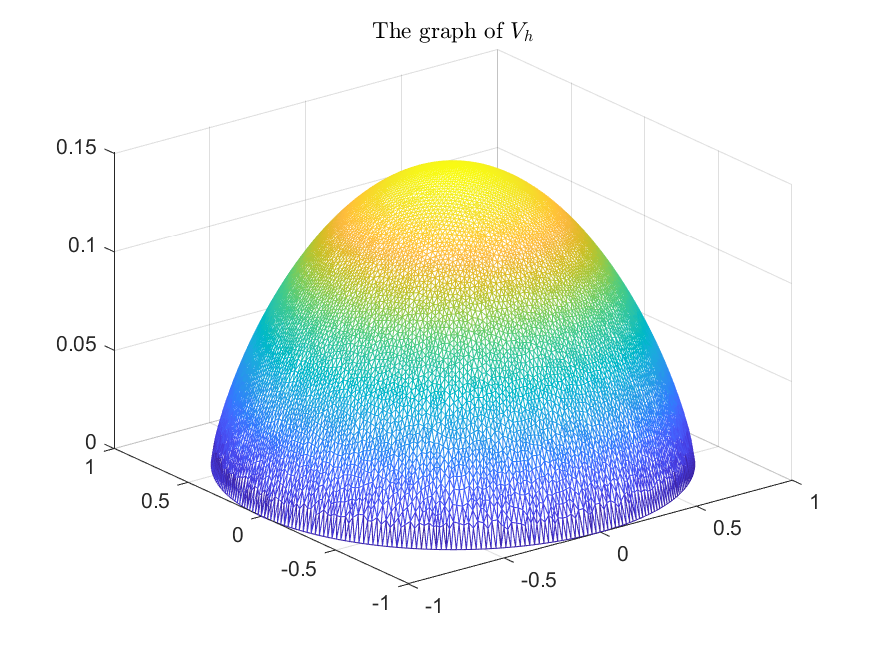}
				\captionof{figure}{	Plots of $U_h$ and $V_{h}$ for $s=0.75$ and $(d_1,d_2,\beta)=(4,3,5)$ at $t=5\times 10^{10}$}
				\label{figure6}
			\end{minipage}
		\end{center}
		
		In addition to the numerical solution plots, we also compute the $l^\infty$-norm of $U_h$ and $V_h$ at $t=5\times 10^{10}$ for different values of $h$. More precisely, we have fixed the time step $h_t$ to $10^{-2}$ and varied the mesh size $h$. The resulting norms are summarized in {\bf Table \ref{table1}}. The purpose of this table is to ensure that the solution does not change significantly with respect to the mesh size. As shown in the table, the norms remain relatively constant as the mesh size is varied. This indicates that the numerical solution is consistent.

		\begin{center}
			\begin{minipage}{\linewidth}
				\centering
				\renewcommand*{\arraystretch}{1.4}
				\begin{tabular}{|c|c|c||c|c|}
					\cline{2-5}
					\multicolumn{1}{c|}{} & \multicolumn{2}{c||}{\footnotesize $d_1=1$,$d_2=2$,$\beta=3$} & \multicolumn{2}{c|}{\footnotesize $d_1=4$,$d_2=3$,$\beta=5$}  \\
					\hline
					h &  $\|U_h\|_\infty$ &  $\|V_h\|_\infty$ & $\|U_h\|_\infty$   & $\|V_h\|_\infty$     \\
					\hline
					1/8\textcolor{white!5}{  } & 0.43189& 0.42518 &    0.09623&    0.15000   \\
					\hline
					1/12\textcolor{white!5}{  } &  0.43221& 0.42623  &   0.09630&    0.15007   \\
					
					\hline
					1/16\textcolor{white!5}{    } &  0.43246& 0.42701  &  0.09635&    0.15012  \\
					\hline
					1/24\textcolor{white!5}{    } & 0.43271& 0.42772 &    0.09637&    0.15016 \\
					\hline
					1/32\textcolor{white!5}{    } & 0.43285& 0.42887 &   0.09643&    0.15022\\
					\hline
					1/48\textcolor{white!5}{    } & 0.43292& 0.42853 &     0.09645&    0.15028 \\
					\hline
				\end{tabular}
				\captionof{table}{	Table of $\ell^\infty$-errors of $U_h$ and $V_h$ for $s=0.75$ at $t=5\times 10^{10}$}
				\label{table1}
			\end{minipage}
		\end{center}
		
		In conclusion, we have studied the global existence of nonnegative numerical solutions to the fractional reaction-diffusion system \eqref{ExampleReactiondiffusionSystem2x2}.
		Our numerical experiments have showed that the system has nonnegative numerical solutions that can be computed for a large final time. This indicates that a global solution in time to System \eqref{ExampleReactiondiffusionSystem2x2} seems to exist. However, it is not clear whether this solution will exist 
		for all times. In future work, we aim at investigating this issue theoretically. 

		\section{Appendix}
		
		
		{In this section, we give the proofs of Theorems \ref{BoundedSol} and \ref{maximumprinciple}. For the convenience of the reader, we recall
			their statements.}
		\begin{manualthm}{3.1} 
			Assume that  $w_0 \in L^\infty(\Omega)$ and $h\in L^p(Q_T$) with $p>1$.
			Let $w$ be the weak solution to Problem \eqref{ParabolicProblem}. Then, for any $p>\dfrac{N+2s}{2s}$, there exists a constant $C>0$ such that
			$$
			\|w\|_{L^\infty(Q_T)}\leq \|w_0\|_{L^\infty(\Omega)}+C\|h\|_{L^p(Q_T)}.
			$$
		\end{manualthm}
		\begin{proof}
		For any $t\geq0$, the unique weak solution to Problem \eqref{ParabolicProblem} reads as :
		\begin{equation}\label{UniqueweakSolution} \tag{$3.4$}
			w(t,\cdot)=S_{A}(t) w_0+\int_0^t S_{A}(t-\tau) f(\tau,\cdot) d\tau.
		\end{equation}   
		As  $\{S_{A}(t)\}_{t\geq 0}$ is $L^\infty$-contractive, we obtain
		$$
		\|w(t,\cdot)\|_{L^\infty(\Omega)}\leq \|w_0\|_{L^\infty(\Omega)}+ \int_0^t \|  S_{A}(t-\tau) f(\tau,\cdot)\|_{L^\infty(\Omega)} d\tau.
		$$	
		According to Theorem \ref{TheoremSemigroup}, the semigroup $\{S_{A}(t)\}_{t\geq 0}$ is ultracontractive. Thus, for $q=+\infty$ and for any $p>1$, there exists a constant $C>0$ such that
		\begin{equation} \tag{$3.5$}
			\|S_{A}(t-\tau)f(\tau,\cdot)\|_{L^\infty(\Omega)}\leq C (t-\tau)^{-\frac{N}{2sp}} \|f(\tau,\cdot)\|_{L^p(\Omega)},\quad \text{for any }\; \tau\in(0,t).
		\end{equation}
		Therefore, we get
		$$
		\|w(t,\cdot)\|_{L^\infty(\Omega)}\leq \|w_0\|_{L^\infty(\Omega)}+C \int_0^T (t-\tau)^{-\frac{N}{2sp}} \|f(\tau,\cdot)\|_{L^p(\Omega)} d\tau.
		$$
		Now, by using H\"{o}lder's inequality, we obtain
		$$
		\|w(t,\cdot)\|_{L^\infty(\Omega)}\leq \|w_0\|_{L^\infty(\Omega)}+C \left(\int_0^T (t-\tau)^{-\frac{N}{2s(p-1)}}d\tau\right)^{\frac{p-1}{p}}  \|f\|_{L^p(Q_T)}.
		$$
		The quantity $\dis\int_0^T (t-\tau)^{-\frac{N}{2s(p-1)}}d\tau$ is finite if and only if $p>\dfrac{N+2s}{2s}$. Hence , we get the desired result.
		\end{proof}
	\bigbreak
	\begin{manualthm}{3.2}	Assume that  $w_0 \in L^\infty(\Omega)$  and $h\in L^p(Q_T$) with $p>1$. Let $w$ be the weak solution to Problem \eqref{ParabolicProblem}. If $h\leq 0$, then $w\in L^\infty(Q_T)$ and 
		$$
		\|w\|_{L^\infty(Q_T)}\leq\|w_0\|_{L^\infty(\Omega)}.	
		$$
	\end{manualthm}
		\begin{proof}
	The idea of the proof is similar to the proof of  \cite[Proposition 3.18]{Fer}.\\
	Let us define $v:=\|w_0\|_{L^\infty(\Omega)}-w$ and write $v=v^++v^-$, where $v^+=\max\{v,0\}$ and $v^-=\min\{v,0\}$. \\
	We will proceed by contradiction. Suppose that the set $\{\mathbf{x}\in\R^N \,:\,v^-(t,\mathbf{x})\neq 0\; \forall t\in (0,T) \}$ is nonempty.
	\\
	Let us multiply the first equation of Problem \eqref{ParabolicProblem} by $v^-$ and integrate over $\R^N$ then over $(0,T)$. Then, we obtain
	$$
	\begin{array}{ll}
		\dis\int_0^T\hskip-3mm \int_\Omega \partial_t w(t, \mathbf{x}) v^-( t,\mathbf{x}) d\mathbf{x} dt+\frac{a_{N,s}}{2} \int_0^T\hskip-3mm \iint_{\R^{N}\times \R^N}\hskip-4mm \frac{(w(t, \mathbf{x})-w(t,\mathbf{y}))(v^-(t,\mathbf{x})-v^-(t, \mathbf{y}))}{\|\mathbf{x}-\mathbf{y}\|^{N+2s}} d\mathbf{x} d\mathbf{y} dt&\\\\&\hspace*{-5.7cm}\dis= \int_0^T\int_\Omega f(t,\mathbf{x})v^-(t, \mathbf{x}) d\mathbf{x} dt.
	\end{array}
	$$
	For any $\mathbf{x},\mathbf{y}\in\R^N$, we have
	$$
	\begin{array}{ll}
		(w(t, \mathbf{x})-w(t, \mathbf{y}))(v^-( t,\mathbf{x})-v^-(t, \mathbf{y}))&\\&\hspace*{-3cm}=(w(t, \mathbf{x})-\|w\|_{L^\infty(\Omega)}+\|w\|_{L^\infty(\Omega)}-w(t, \mathbf{y}))(v^-( t,\mathbf{x})-v^-(t, \mathbf{y}))\\&\hspace*{-3cm}= (-v(t, \mathbf{x})+v(t, \mathbf{y}))(v^-( t,\mathbf{x})-v^-(t, \mathbf{y}))\\&\hspace*{-3cm}= (-v^+(t, \mathbf{x})+v^-(t, \mathbf{x})+v^+(t, \mathbf{y})-v^-(t, \mathbf{y}))(v^-( t,\mathbf{x})-v^-(t, \mathbf{y}))\\&\hspace*{-3cm}= (v^-(t, \mathbf{x})-v^-(t, \mathbf{y}))^2-(v^+( t,\mathbf{x})-v^+(t, \mathbf{y}))(v^-( t,\mathbf{x})-v^-(t, \mathbf{y})).
	\end{array}
	$$
	Moreover, we know that
	if $v^->0$ then $v^+=0$, and if $v^-=0$ then $v^+>0$. Therefore, 
	$$
	v^+(t, \mathbf{y})(v^-(t, \mathbf{x})-v^-(t, \mathbf{y}))\geq 0,
	$$
	This implies
	$$
	(w(t, \mathbf{x})-w(t, \mathbf{y}))(v^-(t, \mathbf{x})-v^-(t, \mathbf{y}))\geq 0.
	$$
	Thence, we get
	$$
	\frac{a_{N,s}}{2} \int_0^T \iint_{\R^{N}\times \R^{N}} \frac{(w(t, \mathbf{x})-w(t, \mathbf{y}))(v^-(t, \mathbf{x})-v^-(t, \mathbf{y}))}{\|\mathbf{x}-\mathbf{y}\|^{N+2s}} d\mathbf{x} d\mathbf{y} dt\geq 0.
	$$
	As $w=\|w_0\|_{L^\infty(\Omega)}-v$, then $\partial_t w=-\partial_t v=-\partial_t v^+ +\partial_t v^-$ and
	$$
	\int_0^T\hskip-3mm \int_\Omega \partial_t w(t,\mathbf{x}) v^-(t, \mathbf{x}) d\mathbf{x} dt=\hskip-2mm  \int_0^T\hskip-3mm  \int_\Omega \partial_t v^-(t, \mathbf{x}) v^-(t, \mathbf{x}) d\mathbf{x} dt\hskip-1mm -\hskip-1mm  \int_0^T \hskip-3mm \int_\Omega \partial_t v^+(t, \mathbf{x}) v^-(t,\mathbf{x}) d\mathbf{x} dt.
	$$
	Furthermore, we have $\dis\int_0^T \int_\Omega\partial_t v^+(t,\mathbf{x}) v^-(t, \mathbf{x}) d\mathbf{x} dt=0$, then
	$$
	\begin{array}{ll}
		\dis\int_0^T \dis\int_\Omega \partial_t w(t, \mathbf{x}) v^-( t,\mathbf{x}) d\mathbf{x} dt&=\dfrac{1}{2}  \dis\int_\Omega \int_0^T \partial_t (v^-(t,\mathbf{x}))^2 dt d\mathbf{x}\\&=\dfrac{1}{2}\dis  \int_\Omega \left( (v^-(T,\mathbf{x}))^2 -(v^-(0,\mathbf{x}))^2    \right) d\mathbf{x}\\&=\dfrac{1}{2}\dis  \int_\Omega \left( v^-(T,\mathbf{x})^2   \right) d\mathbf{x}\geq 0,
	\end{array}
	$$
	as $v^-(0,\mathbf{x})=0$.\\
	Now, we have
	$$
	\int_0^T \hskip-3mm \int_\Omega \partial_t w(t,\mathbf{x}) v^-(t,\mathbf{x}) d\mathbf{x} dt+\frac{a_{N,s}}{2} \int_0^T\hskip-3mm  \int_{\R^{N}\times \R^N}\hskip-6mm \frac{(w(t, \mathbf{x})-w( t,\mathbf{y}))(v^-(t, \mathbf{x})-v^-(t, \mathbf{y}))}{\|\mathbf{x}-\mathbf{y}\|^{N+2s}} d\mathbf{x} d\mathbf{y} dt\geq 0,
	$$
	which is a contradiction as $f\leq 0$ and $\dis\int_0^T\int_\Omega f(t,\mathbf{x})v^-(t,\mathbf{x}) d\mathbf{x} dt\leq 0$. This completes the proof. 
	\end{proof}

\end{document}